\documentclass[a4paper]{amsart}
\usepackage{amsfonts}\usepackage{amssymb}\usepackage{amsmath}\usepackage{amsthm}
\usepackage{stackengine}
\usepackage[utf8]{inputenc}
\usepackage{soul}
\usepackage[T1]{fontenc}
\usepackage{latexsym}
\usepackage{graphicx}
\usepackage{layout}
\usepackage{physics}
\usepackage{subfigure}
\usepackage[normalem]{ulem}
\usepackage{mathtools}  
\mathtoolsset{showonlyrefs}  
\usepackage[shortlabels]{enumitem}
\usepackage{verbatim}
\usepackage{pgfplots}
\usetikzlibrary{arrows,shapes}
\pgfplotsset{compat=1.18}
\usepackage{comment}
\usepackage{scalerel}
\usepackage{xcolor}
\usepackage[colorlinks=true,urlcolor=blue,citecolor=red,linkcolor=blue,linktocpage,pdfpagelabels,bookmarksnumbered,bookmarksopen]{hyperref}
\usepackage{tikz}
\usepackage{tikz-3dplot}
\usetikzlibrary{hobby}
\newtheorem{theorem}{Theorem}[section]
\newtheorem{definition}[theorem]{Definition}

\newtheorem{proposition}[theorem]{Proposition}
\newtheorem{lemma}[theorem]{Lemma}
\newtheorem{corollary}[theorem]{Corollary}
\newtheorem{remark}[theorem]{Remark}

\usepackage{lineno}

\def\R{\mathbb{R}}

\def\S{\mathbb{S}}
\def\N{{\rm I\mskip -3.5mu N}}

\def\di12{\mathcal{D}^{1,2}(\R^n)}

\def\l{{\lambda}}

\def\0l{_{0,\l}}
\def\1l{_{1,\l}}
\def\2l{_{2,\l}}
\def\3l{_{3,\l}}
\def\4l{_{4,\l}}

\def\Om{\Omega}

\def\beq{\begin{equation}}
\def\eeq{\end{equation}}

\def\sideremark#1{\ifvmode\leavevmode\fi\vadjust{\vbox to0pt{\vss
 \hbox to 0pt{\hskip\hsize\hskip1em
 \vbox{\hsize2.1cm\tiny\raggedright\pretolerance10000
  \noindent #1\hfill}\hss}\vbox to15pt{\vfil}\vss}}}%

\newtheorem*{theorem*}{Theorem}

\begin{document}
\title[Lane--Emden Problems on Convex Domains of $\mathbb S^2$]{Geometric Properties of positive solutions of Lane--Emden problems on convex domains of $\S^2$}
\author[Grossi \& Provenzano \& Raom]{Massimo Grossi\ \&\ Luigi Provenzano \&\ Daniel Raom }



\thanks{The first and third author acknowledge support of the INDAM-GNAMPA group. The second author acknowledges support of the INdAM GNSAGA group. The third author gratefully acknowledges the kind hospitality of the Dipartimento di Scienze di Base e Applicate per l’Ingegneria of Sapienza University of Rome during the preparation of the work here presented.}

\address{Massimo Grossi,  Dipartimento di Scienze di Base Applicate per l’Ingegneria, Universit\`a degli Studi di Roma \emph{La Sapienza} P.le A. Moro 5 - 00185 Roma, e-mail: {\sf massimo.grossi@uniroma1.it}.}
\address{Luigi Provenzano,  Dipartimento di Scienze di Base Applicate per l’Ingegneria, Universit\`a degli Studi di Roma \emph{La Sapienza} P.le A. Moro 5 - 00185 Roma, e-mail: {\sf luigi.provenzano@uniroma1.it}.}
\address{Daniel Raom Santiago Bezerra Costa Da Silva,  Dipartimento di Scienze di Base Applicate per l’Ingegneria, Universit\`a degli Studi di Roma \emph{La Sapienza} P.le A. Moro 5 - 00185 Roma, e-mail: {\sf danielraom.santiagobezerra@uniroma1.it}.}
\maketitle

\begin{abstract}
We study positive solutions of the Dirichlet problem  $-\Delta u = u^p$ in a uniformly convex domain $\Om \subset \S^2$, $u= 0$ on $\partial\Om.$ For $p=1$, we assume that the right-hand side is replaced by $\lambda_1 u$, where $\lambda_1$ is the first eigenvalue of $-\Delta$ on $\Om$ with zero Dirichlet boundary condition. We prove that for $0 \leq p < 1$ the unique positive solution $u$ is such that $u^{\frac{1-p}{2}}$ is strictly concave in $\Om$, while for $1 < p \leq 3$ every positive solution $u$ which belongs to the connected component of the least-energy branch of solutions issued from $p=1$ is such that $u^{\frac{1-p}{2}}$ is strictly convex in $\Om.$ \end{abstract}

\section{Introduction and Statement of the Main Result}
Let $\S^2$ be the 2-sphere with its standard metric of constant curvature $1.$ Let $\Om \subset \S^2$ be a domain with smooth boundary and let $u \in C^\infty(\overline{\Om})$ be a solution of the following Dirichlet problem:

\begin{align} \tag{$P_p$} \label{BVP0}
    \begin{cases}
        -\Delta u  = u^p \quad &\text{in} \quad \Om \subset \S^2,\\ 
        \quad \; \, \, u> 0 \quad &\text{in} \quad \Om, \\
        \quad \; \, \, u = 0 \quad &\text{on} \quad \partial \Om,
\end{cases}
\end{align} 
\smallskip
where $-\Delta$ denotes the Laplace-Beltrami operator in $\S^2$ and $0 \leq p< 3.$ 

\smallskip
For $p=1,$ the problem is understood as the Dirichlet eigenvalue problem
    \begin{align} \tag{$P_1$} \label{BVP_eigen}
        \begin{cases}
            -\Delta u_1 = \lambda_1 u_1 \quad &\text{in} \; \; \Om, \\
             \quad \; \;  u_1 > 0 \quad &\text{in} \; \; \Om, \\
            \quad \; \;  u_1 = 0 \quad &\text{on} \; \; \partial\Om, \\
        \end{cases}
    \end{align}
where $\lambda_1$ is the first Dirichlet eigenvalue of $-\Delta$ in $\Om.$ 

\smallskip

For $p=0,$ we have the torsion problem
\begin{align} \tag{$P_0$} \label{BVP1}
    \begin{cases}
        -\Delta u  = 1 \quad &\text{in} \quad \Om \subset \S^2,\\ 
        \quad \; \, \, u> 0 \quad &\text{in} \quad \Om, \\
        \quad \; \, \, u = 0 \quad &\text{on} \quad \partial \Om.
\end{cases}
\end{align}
\smallskip
In this paper, we prove that a solution $u$ of \eqref{BVP0} in a uniformly convex domain $\Om \subset \S^2$ has strictly convex superlevel sets for $0 \leq p \leq 1$, and for $1 <  p \leq 3$ provided that $u$ belongs to a suitable branch of solutions. The conclusion for a solution $u_1$ of problem \eqref{BVP_eigen} was already established in the literature in \cite{LEE-WANG}. We recall the definition of uniform convexity below.  \\

\begin{definition} \label{def_GCdomain}
We say that a smooth domain $\Om \subset \S^2$ is \emph{geodesically convex} if for every pair of points $p, q \in \overline \Om,$ $p \neq q,$ there exists a minimizing geodesic segment connecting $p$ to $q$ that is entirely contained in $\overline \Om.$ If this geodesic segment is also unique and entirely contained in $\Om,$ except possibly at $p$ and $q,$ we say that $\Om$ is \emph{strictly geodesically convex}. If, in addition, there exists a constant $\kappa_0$ such that
    \begin{equation} \label{uniform}
        \kappa_g  \geq \kappa_0 > 0,
    \end{equation}    
\end{definition}
where $\kappa_g$ is the geodesic curvature of $\partial \Om$ with respect to the outward unit normal vector field $\nu$ along $\partial \Om,$ then $\Om$ is said to be \emph{uniformly geodesically convex}.
\smallskip

We shall simply write \emph{strict convexity} in place of \emph{strict geodesic convexity}, and \emph{uniform convexity} in place of \emph{uniform geodesic convexity,} if there is no ambiguity.

\begin{definition}
We say that a smooth function $v : \Om \subset \S^2 \to \R$ is \emph{convex} in $\Om$ if for every $p, q \in \Om$ joined by a unit-speed minimizing geodesic $\gamma : [0,1] \to \Om,$ $s \mapsto \gamma(s)$ with $\gamma(0)=p$ and $\gamma(1)=q$, the composition $v \circ \gamma$ is a convex function in $s$, i.e.
\begin{equation*}
    \frac{d^2}{ds^2}\, (v \circ \gamma)(s)
    = \nabla^2_{\dot{\gamma}(s),\,\dot{\gamma}(s)} v
    \ge 0
    \quad \text{for all } s \in [0,1],
\end{equation*}
where $\nabla^2 v$ is the covariant Hessian matrix. If the above inequality is strict, $v$ is said to be \textit{strictly} convex. Moreover, $v$ is said to be (strictly) \emph{concave} if $-v$ is (strictly) convex.
\end{definition}

In the Euclidean setting, the problem of determining whether the level sets of a solution to a Dirichlet problem inherit the convexity of the underlying domain has been widely investigated in the literature. A first fundamental result in this direction is due to Makar-Limanov \cite{MAKAR-LIMANOV}, who proved that the unique solution of the torsion problem $-\Delta u=1$ in a bounded convex domain $\Om \subset \R^2,$  $u\big|_{\partial \Om} = 0,$ is such that  $\sqrt{u}$ is concave in $\Om.$ Later, Brascamp and Lieb \cite{BRASCAMP-LIEB} established that the first Dirichlet eigenfunction of the Laplacian in a bounded convex domain $\Om \subset \R^n$ is log-concave. In dimension two, Acker, Payne and Phillipin \cite{ACKER-PAYNE-PHILLIPIN} gave a proof of this fact based on the maximum principle. These results were subsequently complemented by Kennington \cite{KENNINGTON}, who studied the sublinear problem $-\Delta u = u^p$ in a bounded convex domain $\Om \subset \R^n,$ $u\big|_{\partial \Om} = 0,$ with $0\leq p<1,$ and proved the concavity of $u^{\frac{1-p}{2}}$; see also \cite{KEADY}. For $p=0,$ this last result generalizes the one of Makar-Limanov to $\R^n.$ \{color{red} For the Lane--Emden superlinear problem $p>1$ in a bounded convex domain of the plane, Lin \cite{LIN} proved the following result.\} 
\begin{theorem} [{Lin \cite[Theorem~2]{LIN}}] \label{theorem_lin}
Let $\Om \subset \R^2$ be a bounded convex domain and let $u$ be the unique least-energy solution of the Dirichlet problem
\begin{align} \tag{$P_p$} \label{BVP_R2}
    \begin{cases}
        -\Delta u  = u^p \quad &\text{in} \quad \Om \subset \R^2,\\ 
        \quad \; \, \, u> 0 \quad &\text{in} \quad \Om, \\
        \quad \; \, \, u = 0 \quad &\text{on} \quad \partial \Om,
\end{cases}
\end{align} 
where $p>1.$ Then $u^{\frac{1-p}{2}}$ is strictly convex in $\Om.$
\end{theorem}
\begin{remark}
The uniqueness of the solution for a bounded convex domain $\Om\subset\R^2$ was proved for $p$ sufficiently  large in \cite{DEMARCHIS-GROSSI-IANNI-PACELLA}, and hence, by Lin's result, the strict convexity of $u^{\frac{1-p}{2}}$ in a convex planar domain holds true for every positive solution when $p>1$ is large enough. However, it remains an open question whether a convex planar domain may admit a positive solution of higher energy for some exponent $p>1.$ Consequently, even in the convex planar case, the strict convexity of $u^{\frac{1-p}{2}}$ is an open problem for every $p>1.$ 
\end{remark}

In all of the aforementioned cases, the solution possesses convex level sets. Moreover, the concavity results recalled above can be sharpened to their strict counterparts. In the planar Euclidean case, the strict convexity of the level sets of the solution $u$ of a Dirichlet problem $-\Delta u = f(u)$ in a bounded convex domain $\Om \subset \R^2,$ $u\big|_{\partial \Om}=0$ is true for some special nonlinearities $f(s),$ as studied by Caffarelli and Friedman \cite{CAFFARELLI-FRIEDMAN}. In fact, the idea exploited in \cite{CAFFARELLI-FRIEDMAN} is to transform the solution $u>0$ by some strictly monotone function $g$ and study the definiteness of the Hessian $D^2 v$ in $\Om.$ One first proves this definiteness in two local regimes: in an Euclidean ball and in a tubular neighborhood of the boundary of a strictly convex domain. Then a deformation of this Euclidean ball into an arbitrary strictly convex domain, together with a constant rank theorem, propagate the positive (or negative) definiteness of $\nabla^2 v$ to the whole domain. Then both $u$ and $g(u)$ possess strictly convex superlevel sets (see e.g. \cite{CAFFARELLI-FRIEDMAN}, Corollary 3.2). The main technique devised in \cite{CAFFARELLI-FRIEDMAN}, which came to be known as a "structure theorem" or a "constant rank theorem", was then generalized to higher dimensions by Korevaar and Lewis in \cite{KOREVAAR-LEWIS}. We should note, however, that there exist bounded convex domains $\Om \subset \R^2$ and $C^{\infty}$ functions $f:[0, + \infty) \to \R$ such that the problem $-\Delta u = f(u)$ in $\Om$ with zero Dirichlet boundary condition admits a positive solution whose superlevel sets are not convex, as established by Hamel, Nadirashvilii and Sire \cite{HAMEL-NADIRASHVILI-SIRE}. 
\medskip

In the spherical setting, an equally satisfactory theory is not yet available, at least not in the spirit of Caffarelli–Friedman \cite{CAFFARELLI-FRIEDMAN} and Korevaar–Lewis \cite{KOREVAAR-LEWIS}. Nevertheless, some positive results are known in special situations. For the first Dirichlet eigenfunction $u_1$ of $-\Delta$ in a strictly convex domain $\Om \subset \S^n,$ Lee and Wang proved in \cite{LEE-WANG} that $\log u_1$ is strictly concave (see Proposition \ref{log_eigen}). Their argument adapts the deformation method of Caffarelli and Friedman \cite{CAFFARELLI-FRIEDMAN} to the spherical case, based on the deformation of a geodesic ball into an arbitrary strictly convex domain. Although not explicitly stated in \cite{LEE-WANG}, the strict log-concavity of the first eigenfunction $u_1$ yields the strict convexity of its superlevel sets. More recently, it has been shown by Khan, Saha and Tuerkoen in \cite{KHAN} that in a convex domain $\Om \subset \S^2$ having a circumradius of at most $2 \arctan 1/5,$ the torsion function in $\Om$ (the solution of problem \eqref{BVP1}) has convex superlevel sets. This shows that already for the torsion problem, establishing the convexity of the level sets of a solution to an elliptic problem in a spherical setting is far from being trivial.   \\

\subsection{Statement of the main result}

Regarding problem \eqref{BVP0}, a classical solution $u \in C^2({\Om}) \cap C(\overline{\Om})$ exists for each $p \geq 0,$ and by the regularity of the boundary, $u \in C^\infty(\overline{\Om})$ by a standard bootstrap argument.  

In the particular case where $\Om \subset \S^2$ is a geodesic ball of radius $0<R < \pi/2,$ it is known that for the Dirichlet problem $-\Delta u = f(u)$ in $\Om,$ $u\big|_{\partial \Om} = 0$ where $f$ is of class $C^1,$ any solution $u \in C^2(\overline{\Om})$ is radially symmetric and strictly decreasing in the radial variable. This fact was shown by Kumaresan and Prajapat in \cite{KUMARESAN-PRAJAPAT} by adapting the classical moving planes technique of Gidas, Ni and Nirenberg \cite{GIDAS-NI-NIRENBERG} to the spherical case. The nonlinearities $f(s) = s^p$ with $p \neq 1$ and $f(s) = \lambda_1 s$ fall under this setting, hence the level sets of the solution of \eqref{BVP0} are concentric geodesic balls.

In the case where $\Om \subset \S^2$ is an arbitrary uniformly convex domain, to the best of the authors' knowledge there are no general results in the literature concerning the convexity of the level sets of a solution of the Dirichlet problem \eqref{BVP0}. In this work we provide a positive answer for $0 \leq p \leq 1,$ and for a suitable branch of solutions for $1 < p \leq 3$.
\medskip

\begin{theorem} \label{power}
    Let $\Om \subset \S^2$ be a \emph{uniformly convex} domain with smooth boundary and let $u: \overline{\Om}\to \R$ be a solution of the Dirichlet problem \eqref{BVP0}. Then, 
    \begin{enumerate}[label=(\alph*)]
        \item (\textit{Sublinear problem}) $u^{\frac{1-p}{2}}$  is strictly concave for $0 \leq p < 1;$ \label{sublinear} 
       	\item (\textit{Eigenvalue problem}) $\log u$ is strictly concave for $p=1$ (Lee--Wang \cite{LEE-WANG}); \label{eigenvalue}
\item  (\textit{Superlinear problem}) $u^{\frac{1-p}{2}}$ is strictly convex
  for $1<p\leq 3$ provided that $(p,u)$ belongs to the connected component
  containing the least-energy branch issued from $p=1$. \label{superlinear}
    \end{enumerate}
\end{theorem}

\begin{corollary} \label{corollary_power}
For every $0 \leq p \leq 1,$ any solution $u$ of the Dirichlet problem \eqref{BVP0} has a unique, non-degenerate critical point (a maximum), and strictly convex superlevel sets. For every $1 < p \leq 3$, any solution $u$ of the Dirichlet problem \eqref{BVP0} in the connected component of the least-energy branch of solutions issued from $p=1$ has a unique, non-degenerate critical point (a maximum), and strictly convex superlevel sets.
\end{corollary}

\bigskip
Below we make a few comments on Theorem \ref{power} and Corollary \ref{corollary_power}.

\begin{remark}\label{rmk_mere_convexity}
If the uniform convexity hypothesis on $\Om \subset \S^2$ is weakened to mere convexity, then similar conclusions to the ones given in Theorem \ref{power} and Corollary \ref{corollary_power} may be obtained. We discuss this in Section \ref{section_generalization}.
\end{remark}

\begin{remark}
For $p=0,$ Theorem \ref{power} \ref{sublinear} gives the spherical analogue of the classical $1/2-$concavity of the torsion function in a bounded convex domain $\Om \subset \R^2;$ see \cite{MAKAR-LIMANOV}. In the spherical setting, both the strict $1/2$-concavity of the torsion function and the strict convexity of its superlevel sets appear to be novel results, and are here obtained without making any assumption on the circumradius of $\Om \subset \S^2.$ By Corollary \eqref{corollary_power}, the uniqueness and nondegeneracy of the critical point of the torsion function follow from this stronger result. These two properties have been recently proven to hold for the torsion function in convex domains $\Om \subset \S^2$ via different techniques, in a work by the first two authors and Gladiali \cite{GLADIALI-GROSSI-PROVENZANO}; see also \cite{GROSSI-PROVENZANO}.
\end{remark}

\begin{remark}
For $0 < p <1,$ the conclusion in Theorem \ref{power} \ref{sublinear} is analogous to the one obtained by Keady \cite{KEADY} for the same equation in a bounded convex domain $\Om \subset \R^2,$ using techniques following \cite{ACKER-PAYNE-PHILLIPIN} and \cite{MAKAR-LIMANOV}.
\end{remark}

\begin{remark}
For $p=1,$ the conclusion of Theorem \ref{power} \ref{eigenvalue} is well-known and holds also in higher dimensions \cite{LEE-WANG}. We recall this result in Proposition \ref{log_eigen}. 
\end{remark}
\begin{remark}
In the spherical setting, the uniqueness of the positive solution to the problem \eqref{BVP0} in a geodesically convex domain $\Om \subset \S^2$ for every $p>1$ is also open. Following the idea of Lin in \cite[Lemma 3]{LIN}, here we show that this is true when $p>1$ is sufficiently close to 1 (thus it is the least-energy solution); see Theorem \ref{theorem_uniqueness}. For $p>1$ sufficiently large, we are unaware of a uniqueness result for this case, for instance in the spirit of \cite{DEMARCHIS-GROSSI-IANNI-PACELLA}. In fact,  the uniqueness of the least-energy solution for every $p>1$ (and even for every $1 < p \leq 3$) remains open in the case of convex domains of $\S^2.$ To the best of our knowledge, no result in the likes of the nondegeneracy result in \cite[Lemma 2]{LIN} is available in the spherical case.

Nevertheless, the conclusion in Theorem \ref{power} \ref{superlinear} holds true for any connected component which contains the least-energy branch of solutions issued from $p=1$. Note that this means that the strict convexity there stated persists along any bifurcation issued from this branch, as well as for subsequent bifurcations, even if a solution in this connected component ceases to be of least-energy. For a branch which is disconnected to the least-energy branch issued from $p=1,$ our continuation argument can no longer be employed.  
\end{remark}

\begin{remark}
In contrast to the planar result of Theorem \ref{theorem_lin}, while in Lemma \ref{lemma_boundary_rank} we obtain that for every $p>1$ the strict convexity of $u^{\frac{1-p}{2}}$ holds in a small tubular neighborhood $\Om_\delta$ of the boundary of a uniformly convex domain $\Om \subset \S^2,$ for $p>3$ we cannot guarantee that this holds true in the whole domain. In fact, such a restriction comes from the constant rank computations; see equation \eqref{Deltaphi_finalestimate}. It is not clear to us if the assumption $1 \leq p \leq 3$ presents itself as intrinsic or merely technical, which begs us to pose:\medskip

\textbf{Open problem.} Let $\Om \subset \S^2$ be a uniformly convex domain with smooth boundary and $u$ be a solution of \eqref{BVP0}. Are the superlevel sets of $u$ strictly convex for $p>3?$
\end{remark}

\smallskip

\subsection{Strategy of the proof}

In general, it is too strong to expect a solution $u$ of the Dirichlet problem \eqref{BVP0} to be strictly convex or strictly concave in a uniformly convex domain $\Om \subset \S^2.$ One may then ask whether this is true about $g(u)$, for some appropriate function $g:(0, \infty) \to \R.$ Following this idea, we shall consider

\begin{align}
&g(s) = s^{\frac{1-p}{2}} \quad \text{if } 0\leq p < 1 \; \text{or} \; 1< p \leq 3,  \label{g_sub} \\
\end{align} 

By setting $v:= g(u)= u^{\frac{1-p}{2}},$ observe that $v$ satisfies the equation

\begin{align}
\Delta v = \frac{1}{v} \left( -\frac{1+p}{1-p} |\nabla v|^2 - \frac{1-p}{2}\right) \quad \text{in } \Om \subset \S^2.  \label{deltav_sub}
\end{align}

\medskip

Throughout the paper, unless otherwise specified, we shall always consider that

\begin{equation} \label{defomega}
\boxed{
    \Om \subset \S^2 \, \text{ is a \emph{uniformly convex} domain with smooth boundary $\partial\Om$}.
    }
\end{equation}
\smallskip

The proof of Theorem \ref{power} \ref{sublinear}, \ref{superlinear} uses the following four ingredients:
\begin{enumerate} [(I)]
    \item The strict log-concavity of the first Dirichlet eigenfunction $u_1$ of $-\Delta$ in a uniformly convex domain $\Om \subset \S^2,$ normalized so that $\max_\Om u_1 = 1;$ see Proposition \ref{log_eigen}. \label{ing1}
    \item The uniform convergence $\frac{u}{||u||_\infty} \to u_1$ as $p \to 1$ in $\Om \subset \S^2,$ where $u$ is a solution of \eqref{BVP0} and $u_1$ is the normalized first Dirichlet eigenfunction from \ref{ing1}; see Corollaries \ref{corollary_sub} and \ref{corollary_super}. \label{ing2}
    \item The full rank of the covariant Hessian of $u^{\frac{1-p}{2}}$ in a tubular neighborhood of the boundary $\partial \Om$; see Lemma \ref{lemma_boundary_rank}. \label{ing3}
  \item A constant rank theorem for the covariant Hessian of $u^{\frac{1-p}{2}}$; see Theorem \ref{MAIN}. \label{ing4}
\end{enumerate} 

The strategy of the proof of Theorem \ref{power} \ref{sublinear}, \ref{superlinear} goes by fixing the domain and employing a continuation argument in the exponent $p,$ rather than deforming a geodesic ball into a uniformly convex domain as in the classical deformation method (e.g. \cite{CAFFARELLI-FRIEDMAN}, \cite{LEE-WANG}). Here we describe the superlinear case $1<p\leq3.$ By \ref{ing1} and \ref{ing2}, provided the domain $\Om \subset \S^2$ is uniformly convex and the exponent $p>1$ is sufficiently close to 1, problem \eqref{BVP0} has a unique solution $u$ such that $v=u^{\frac{1-p}{2}}$ is strictly convex in $\Om.$ We then consider $p>1$ to be the first exponent for which this strict convexity fails in $\Om,$ so that there is some point $x_0$ for which the covariant Hessian matrix of $v,$ denoted $\nabla^2 v,$ has rank 1. By \ref{ing4}, $\Delta v$ (see equation \eqref{deltav_sub}) has the suitable structure to guarantee that $\nabla^2 v$ has constant rank 1 in $\Om.$ But by \ref{ing3}, $\nabla^2 v$ has full rank near the boundary, thus forcing $\nabla^2 v$ to have rank 2 at $x_0,$ which is a contradiction. Hence $u^{\frac{1-p}{2}}$ is strictly convex in $\Om$ for every $1<p\leq 3.$ The sublinear case $0 \leq p < 1$ is analogous.

\smallskip

\subsection{Organization of the paper} \label{subsection_organization}
The present paper is divided as follows:  \\
• In Section \ref{section_preliminaries} we recall some geometric notions, present some background concepts on the stereographic projection (which will help us in the proof of Theorem \ref{theorem_uniqueness}), and recall the $\log$-concavity result for the first eigenfunction \cite{LEE-WANG} (Proposition \ref{log_eigen}). \\
•  In Section \ref{section_uniqueness} we show that, for some $p_0 > 1,$ problem \eqref{BVP0} admits at most one positive solution for each $0 \leq p <p_0,$ provided $\Om\subset \S^2$ is uniformly convex. For $p$ in this range, we also prove the uniform convergence mentioned in \ref{ing2}.  \\
• In Section \ref{section_concavity_boundary} we consider the transformation $v:=u^{\frac{1-p}{2}}$ and prove, in Lemma \ref{lemma_boundary_rank}, that the covariant Hessian $\nabla^2 v$ has full rank in a tubular neighborhood of $\partial\Om$. \\
• In Section \ref{section_constant_rank} we prove the main tool used in the proof of Theorem \ref{power}, namely a constant rank theorem for the covariant Hessian of $u^{\frac{1-p}{2}}.$\\
• In Section \ref{section_proof_power} we finally prove Theorem \ref{power} and Corollary \ref{corollary_power}.\\
• In Section \ref{section_generalization} we consider some directions of further investigation: the relaxation of the uniform convexity assumption (as mentioned in Remark \ref{rmk_mere_convexity}) and the generalization of the results here presented to $\S^n.$

\bigskip
{\em Acknowledgment} The authors would like to express their gratitude to Daniele Bartolucci for his insightful remarks and suggestions concerning the result in Theorem \ref{power} \ref{superlinear}.
\bigskip

\section{Preliminaries} \label{section_preliminaries}
In this Section we recall some basic notions and present some preliminary results that will be used in the proofs of our results.

\subsection{Geometric framework and conventions}

We consider the unit sphere $\S^2 = \{ (x,y,z) \in \R^3 :
x^2 + y^2 + z^2 = 1  \} \subset \R^{3}$ endowed with the standard round metric $g$ of constant curvature $1$, namely the metric induced by the Euclidean metric $g_E$ on $\R^3$ through the canonical inclusion $\S^2 \hookrightarrow \R^3.$ We shall write $\S^2$ in place of $(\S^2, g)$ without further mention.

For $p,q\in\S^2$, we denote by $d (p,q)$ the geodesic distance, and by $d (q,\partial\Om):=\inf_{p \in \partial \Om} d(p,q)$ the distance of a point to the boundary of $\Om$. Moreover, for each $p \in \partial \Om,$ we denote by $\nu$ the outward unit normal to $\Om$ at $p,$ and by $\tau$ the unit tangent vector to $\partial \Om$ at $p$ chosen so that $\{ \tau, \nu, N \}$ is a positively oriented basis of $\R^3.$ Here $N=(x,y,z)$. We shall adopt the convention $\kappa_g:= - \langle \nabla_\tau \tau, \nu \rangle$ for the geodesic curvature of $\partial \Om$ with respect to $\nu,$ so that in accordance to Definition \ref{def_GCdomain}, a uniformly convex domain $\Om \subset \S^2$ has strictly positive geodesic curvature.

Throughout our exposition we shall work with the smooth function $v:\Om\subset \S^2\to \R$ given by $v:=g(u),$ where $g(s)=s^{\frac{1-p}{2}}.$ An important object used throughout our analysis is the covariant Hessian of $v$, denoted by $\nabla^2 v,$ which is the symmetric bilinear form on the tangent bundle $T\Om$ defined by $\nabla^2_{X,Y} v := \nabla^2 v (Y,X) = Y(Xv) - (\nabla_Y X)v,$ where $X,Y$ are smooth tangent vector fields on $\Om.$ Relative to a local orthonormal frame $\{e_k\}, k=1,2$, it is represented by the symmetric matrix $\left(\nabla^2 v\right)_{ij} := \left(\nabla^2 v(e_i,e_j)\right).$  Although this matrix depends on the chosen frame, its rank and definiteness are intrinsic, i.e. coordinate independent. Since our arguments involve only these quantities, we shall often identify $\nabla^2 v$ with its matrix representation $(\nabla^2 v)_{ij}$ with respect to a convenient local orthonormal frame $\{e_k \}, k = 1,2,$ without risk of confusion. 

\subsection{Log-concavity of the first eigenfunction}
The following result constitutes the first ingredient for the proof of Theorem \ref{power}, as explained in Subsection \ref{subsection_organization}.
\medskip

\begin{proposition}  [Lee--Wang \cite{LEE-WANG}] \label{log_eigen}
    Let $\Om \subset \S^n$ be a strictly convex domain with smooth boundary, and let $u_1$ be a solution of the Dirichlet problem \eqref{BVP_eigen}. Then $\log u_1$ is strictly concave in $\Om.$
\end{proposition}

\medskip

\subsection{Stereographic projection} \label{subsection_stereo}
In Theorem \ref{power} we will prove that the uniqueness of the solution of problem \eqref{BVP0} holds in a convex domain $\Om \subset \S^2,$ provided the exponent $p$ is sufficiently close to $1.$ To this end, it is convenient to identify $\S^2 \setminus \{q\}$ with $\R^2$ via stereographic projection, so that $\R^2$ is endowed with a conformal metric.

\medskip

Let $q$ be any point of $\S^2$ (e.g. the north pole). Then $\S^2\setminus\{q\},$ endowed with the standard round metric $g,$ is isometric to $\R^2$ endowed with the conformal metric $\tilde{g} := \frac{4}{(1+X^2+Y^2)^2}(dX^2+dY^2).$ Here $(X,Y)$ are the usual cartesian coordinates in $\R^2$. Let then $\Om \subset \S^2$ be a convex domain. By Definition \ref{def_GCdomain}, we may assume that $\Om$ is contained in an open hemisphere, with the antipode $-q \in \Om$ (e.g. the south pole). Then $\Om \subset \S^2$ with the round sphere metric $g$ is isometric to $\widetilde\Om\subset\R^2$ with the conformal metric $\widetilde{g},$ and $\widetilde \Om$ is contained in the disk of radius $1$ centered at the origin; see Corollary \ref{corollary_isometry} below. This first lemma relates the curvature of $\partial \widetilde\Om \subset \R^2$ for the conformal metric $\tilde{g}$ (which is also the curvature of $\partial \Om\subset\S^2$) with the curvature of $\partial \widetilde\Om$ for the Euclidean metric $g_E.$
\begin{lemma} \label{lemma_curvatures_conformal_euclidean}
Let $\widetilde\Om\subset\R^2$ be a smooth bounded domain.  Suppose that $\widetilde\Om$ has boundary curvature $\tilde{\kappa}$ with respect to the metric $\tilde{g} = \frac{4}{(1+X^2+Y^2)^2}(dX^2+dY^2),$ and boundary curvature $\kappa_E$ with respect to the Euclidean metric $g_E = dX^2+dY^2$. Then
\begin{equation} \label{kappas_can_stereo}
\tilde{\kappa}=\frac{1+X^2+Y^2}{2}\kappa_E-\langle (X,Y),\nu_E\rangle,
\end{equation}
where $\nu_E$ is the outward unit normal to $\widetilde\Om$ for the Euclidean metric and $\langle\cdot,\cdot\rangle$ is the  Euclidean scalar product.
\end{lemma}
\begin{proof}
The proof follows from the following well-known fact: if $(M, g)$ is a $2$-dimensional Riemannian manifold,  $\Om \subset M$ is a smooth domain with outward unit normal $\nu$ (for the metric $g$), $\kappa_g$ is the curvature of $\partial\Om$ with respect to $\nu$, and $\rho>0$, then $\kappa_{\rho^2g}=\frac{1}{\rho}(\kappa_g-\partial_{\nu}\log(\rho))$, where $\kappa_{\rho^2g}$ is the curvature of $\partial\Om$ in the conformal metric $\rho^2g$. Then taking $(M,g) = (\R^2, g_E)$ and $\rho=\frac{2}{1+X^2+Y^2}$ gives \eqref{kappas_can_stereo}.  
\end{proof}

\smallskip

From Lemma \ref{lemma_curvatures_conformal_euclidean} we deduce this next Corollary:

\smallskip

\begin{corollary} \label{corollary_isometry}
    Let $\Om\subset\S^2$ be a convex domain with smooth boundary. Then $\Om$ is isometric to a domain $\widetilde\Om\subset\R^2$ (for the conformal metric $\tilde g$) which is contained in the disk of radius $1$ centered at the origin and is convex with respect to the Euclidean metric. Moreover, if $\Om \subset \S^2$ is uniformly convex, then $\widetilde \Om \subset \R^2$ is uniformly convex with respect to the Euclidean metric.
\end{corollary}

Then each of the problems \eqref{BVP0} and \eqref{BVP_eigen} can be reformulated as an equivalent problem on a domain of $\R^2$. Let $\Om\subset \S^2$, and let $\widetilde\Om\subset\R^2$ be such that, for the conformal metric $\widetilde g = \frac{4}{(1+X^2+Y^2)^2}(dX^2+dY^2),$ it is isometric to $\Om$. Call $\pi:\Om \to \widetilde \Om$ this isometry (it is a stereographic projection). Then we have the following: 

\begin{lemma} \label{lemma_stereo}
Let $u$ be a solution of \eqref{BVP0} in $\Om\subset\S^2$. Then $\tilde u=u\circ\pi^{-1}$ solves
\begin{align} \label{eq_stereo}
\begin{cases}
\; \; \, -\Delta\tilde u=\displaystyle\frac{4}{(1+X^2+Y^2)^2}\tilde u^p & {\rm in\ }\widetilde\Om \subset \R^2,\\
\quad \, \; \; \; \; \tilde u>0 & {\rm in\ } \tilde\Om , \\
\quad \, \; \; \; \; \tilde u=0 & {\rm on\ }\partial\tilde\Om , \\
	     \end{cases}
\end{align}
where $-\Delta$ is the Euclidean Laplacian.
\end{lemma}

\begin{lemma} \label{lemma_stereo_eigen}
Let $u_1$ be the solution of \eqref{BVP_eigen} in $\Om \subset\S^2,$ normalized so that $\max_\Om u_1 = 1.$ Then $\tilde u_1=u_1\circ\pi^{-1}$ solves
 \begin{align} \label{BVP_eigen_R2}
        \begin{cases}
           \; -\Delta \tilde u_1 = \displaystyle\frac{4}{(1+X^2+Y^2)^2}  \lambda_1 \tilde u_1 \quad &\text{in} \; \; \widetilde \Om \subset \R^2, \\
           \quad \; \; \; \tilde u_1 > 0 \quad &\text{in} \; \; \widetilde \Om,\\
            \quad \; \; \; \tilde u_1 = 0 \quad &\text{on} \; \; \partial \widetilde \Om, \\
        \end{cases}
  \end{align}
where $-\Delta$ is the Euclidean Laplacian. Moreover, $\max_{\widetilde \Om} \tilde u_1 =1.$
\end{lemma}

Lemmas \ref{lemma_stereo} and \ref{lemma_stereo_eigen} are just a consequence of the fact that, in dimension $2,$ the Laplacian in a conformal metric $\tilde g$ is simply the Laplacian in the original metric $g$ divided by the conformal factor, that is, $\Delta_{\tilde g} = \rho^{-2} \Delta_g.$

\medskip

\subsection{A classical uniqueness result for a sublinear Dirichlet problem}

In order to treat the sublinear case $0\leq p<1$ in the proof of Theorem \ref{theorem_uniqueness}, we recall this next classical result, which holds for any bounded domain of $\R^2.$

\begin{proposition} [Brezis and Oswald \cite{BREZIS-OSWALD}] \label{brezis_oswald}
    Let $\widetilde{\Om } \subset \R^2$ be a bounded domain with smooth boundary and consider the Dirichlet problem
    \begin{align}\label{sub_brezis_oswald}
    \begin{cases} 
            -\Delta \tilde{u} = \rho^2(x) f(\tilde{u}) \quad &\text{in} \; \widetilde{\Om} \subset \R^2, \\
             \quad \; \; \, \tilde{u} > 0 \quad &\text{in} \; \widetilde{\Om}, \\
             \quad \; \; \, \tilde{u}=0 \quad &\text{on} \; \partial\widetilde{\Om},
    \end{cases}
    \end{align}
    with $f:\widetilde{\Om} \to \R$ such that $s \mapsto \displaystyle\frac{f(s)}{s}$ is decreasing and $\rho^2(x) \geq 0,$ $\rho \not \equiv 0,$ for each $x \in \widetilde \Om.$ 
    Then there is at most one solution of \eqref{sub_brezis_oswald}.
\end{proposition}
We refer to (\cite{BREZIS-KAMIN}, Appendix II) or (\cite{BREZIS-OSWALD}, Theorem 1) for a proof of Proposition \ref{brezis_oswald}. We point out that in the proof of Theorem \ref{theorem_uniqueness} we shall take $\rho(x) = \frac{2}{1+|x|^2},$ where $x=(X,Y)\in \R^2.$

\medskip

\section{Uniqueness of the solution for $p$ close to 1 in a convex domain}  \label{section_uniqueness}

This section is devoted to the proofs of Theorem \ref{theorem_uniqueness} and Corollary \ref{corollary_sub}, given below. Besides their intrinsic interest, these results provide us with the ingredient \ref{ing2} needed for the proof of Theorem \ref{power}. Throughout the present section and without further mention, we consider that
\begin{equation} \label{normalized_eigen}
\boxed{
\text{ $u_1$ is the solution of \eqref{BVP_eigen} normalized so that $\rVert u_1 \lVert_\infty = \max_\Om u_1 = 1.$ }
}
\end{equation}

In the superlinear case, if $p=p_n$ with $p_n\downarrow1$ and $u_n := u_{p_n}$ denotes the corresponding solution of \eqref{BVP0}, then $\frac{u_n}{\lVert u_n \rVert_\infty}\to u_1$ uniformly in $\Om.$  The corresponding convergence in the sublinear case $p_n\uparrow1$ is obtained in Corollary \ref{corollary_sub}. The argument in Theorem \ref{theorem_uniqueness} goes by first translating the problem on $\S^2$ into an equivalent weighted Dirichlet problem on $\R^2$ (see equation \eqref{eq_stereo}), where the uniqueness argument is more tractable. Then for $0\leq p<1$, uniqueness follows directly from the classical result of Proposition \ref{brezis_oswald}, whilst for $1 < p \leq p_0$ we adapt an argument of Lin \cite{LIN}.

Let us also remark that even though problem \eqref{BVP0} may admit multiple positive solutions for $1 < p \leq 3,$ later in the proof of Theorem \ref{power} we shall only use the uniqueness in the limiting regime $p \to 1;$ see Corollaries \ref{corollary_sub} and \ref{corollary_super} below.  

\bigskip

\begin{theorem} \label{theorem_uniqueness}
    Suppose that $\Om \subset \S^2$ is a convex domain with smooth boundary. Then for each fixed $0 \leq p <1,$ the Dirichlet problem \eqref{BVP0} admits a unique positive solution $u.$ Moreover, there exists some $p_0>1$ such that, for each fixed $1 < p \leq p_0,$ the Dirichlet problem \eqref{BVP0} admits a unique positive solution $u.$ 
\end{theorem}

\medskip

\begin{corollary} \label{corollary_sub} 
Let $(p_n)_{n \in \N} \subset [0,1)$ with $p_n \uparrow 1$ as $n \to \infty.$ For each $n \in \N,$ let $u_n:= u_{p_n}$ be the corresponding unique positive solution of \eqref{BVP0}, with $p=p_n,$ in a convex domain $\Om \subset \S^2.$ Then $\displaystyle \frac{u_n}{\lVert u_n \rVert_{\infty}} \to u_1$ uniformly in $\Om$ as $p_n \uparrow 1.$
\end{corollary}

\begin{corollary} \label{corollary_super} 
Let $(p_n)_{n \in \N} \subset (1,p_0]$ with $p_n \downarrow 1$ as $n \to \infty.$ For each $n \in \N,$ let $u_n:= u_{p_n}$ be the corresponding unique positive solution of \eqref{BVP0}, with $p=p_n,$ in a convex domain $\Om \subset \S^2.$ Then $\displaystyle \frac{u_n}{\lVert u_n \rVert_{\infty}} \to u_1$ uniformly in $\Om$ as $p_n \downarrow 1.$
\end{corollary}

\bigskip

\begin{remark}
The convexity assumption in Theorem \ref{theorem_uniqueness} is likely an artifact of the proof and could be removed. This is the case for each fixed $0 \leq p \leq 1,$ thanks to Theorem \ref{brezis_oswald} and the well-known uniqueness of the first eigenfunction up to normalization. Although our result uses the convexity the domain, we believe the uniqueness also holds for each fixed $1 < p \leq p_0.$
\end{remark}

\smallskip

\begin{proof} [Proof of Theorem \ref{theorem_uniqueness}]
First of all, by Corollary \ref{corollary_isometry} we have that $\Om$ is isometric to a domain $\widetilde \Om \subset \R^2$ endowed with the conformal metric $ \frac{4}{(1+X^2+Y^2)^2}(dX^2+dY^2)$ via stereographic projection $\pi: \Om \to \widetilde \Om.$ Moreover, $\widetilde \Om$ is a smooth bounded convex domain for the Euclidean metric $g_E.$ By Lemma \ref{lemma_stereo}, $\tilde u =  u \circ \pi^{-1} $ is a solution of \eqref{eq_stereo}. By Lemma \ref{lemma_stereo_eigen}, $\tilde u_1= u_1 \circ \pi^{-1}$ is a solution of \eqref{BVP_eigen_R2} with $\max_{\widetilde \Om} \tilde u_1 = 1.$

\medskip

We treat the proof of the cases $0\leq p<1$ and $1<p\leq p_0$ separately. For $p=1$ the uniqueness of the positive solution of \eqref{BVP_eigen} (up to scaling) is a well-known fact.\\
    
• $0\leq p<1.$ \\
Fixed an exponent $p$ on this interval, we observe that \eqref{eq_stereo} is in the form of equation \eqref{sub_brezis_oswald}, where $\rho^2(x) = \frac{4}{(1+|x|^2)^2} > 0$ for every $x = (X,Y) \subset \R^2,$ and $f(s) = s^p$, $0 \leq p < 1$ is such that $s \mapsto f(s)/s$ is decreasing. Therefore, by Proposition \ref{brezis_oswald} we have that $\tilde{u}$ is the unique positive solution of problem \eqref{sub_brezis_oswald}. In addition, $\tilde{u} \in L^\infty(\widetilde{\Om}).$ By the smoothness of the isometry $\pi,$ it follows that $u = \tilde u \circ \pi \in L^\infty(\Om)$ is the unique positive solution of problem \eqref{BVP0}.\\

\smallskip

• $ 1<p\leq p_0$ \\
\smallskip
We adapt the proof in (\cite{LIN}, Lemma 3) by considering the weight $\rho^2(x) = \frac{4}{(1+ |x|^2)^2},$ and follow these next steps. \\ 
\smallskip

\textit{Step 1. If $\tilde u_1 \not \equiv \tilde u_2$ are two solutions of \eqref{eq_stereo}, then $\tilde u_1- \tilde u_2$ changes sign.} \\

From the equation \eqref{eq_stereo} and using Green's identity, we compute
\begin{equation}
  0=  \int_{\partial \widetilde  \Om}  \tilde u_2 \frac{\partial \tilde u_1}{\partial \nu} - \tilde u_1 \frac{\partial \tilde u_2}{\partial \nu} =  \int_{\widetilde \Om} \tilde u_2 \Delta \tilde u_1 - \tilde u_1 \Delta \tilde u_2 = \int_{\widetilde \Om} \rho^2 \tilde u_1 \tilde u_2 (\tilde u_1^{p-1} - \tilde u_2^{p-1}).
\end{equation}
Suppose without loss of generality that $\tilde u_1 \geq \tilde u_2,$ which implies $\tilde u_1^{p-1} - \tilde u_2^{p-1} \geq 0.$ The integrand on the RHS must vanish, hence $\tilde u_1 \equiv \tilde u_2$ since $\rho^2 > 0$ and $\tilde u_1, \tilde u_2 >0.$ \\

\smallskip

\textit{Step 2. $M_n^{p_n - 1} = \big( \sup_{\overline{\widetilde \Om}} \tilde u_n  \big)^{p_n - 1}$ is bounded.}\\

Let $(p_n)_{n \in \N} \subset (1, p_0]$ be a sequence such that $p_n \downarrow 1,$ and for each $n \in \N,$ let $\tilde u_n$ be the corresponding solution of the equation \eqref{eq_stereo} with $p=p_n$. By standard elliptic regularity theory we have that $\tilde u_n \in L^\infty(\Om),$ hence we may set
\begin{equation} \label{aux_Qn}
    M_n := \sup_{\overline{\widetilde \Om}} \tilde u_n = \tilde u_n(Q_n) \quad \text{for some } Q_n \in \widetilde \Om.
\end{equation}

First we show that $M_n^{p_n-1}$ is bounded from below for $p_n>1$ sufficiently close to $1.$ Indeed, note $0<\rho^2_{min}:=\frac{4}{(1+R^2)^2} \leq \rho^2(x) \leq 4 =: \rho^2_{max}$ in $\widetilde \Om,$ where $R:= \sup_{x \in \widetilde \Om} |x|.$ Also, let $(\lambda_1,\tilde u_1)$ be the first Dirichlet eigenpair of the Laplacian $-\Delta$ in $\widetilde \Om$ with weight $\rho^2,$ i.e. $\tilde u_1$ is the solution of the problem \eqref{BVP_eigen_R2}.
By multiplying the equation $-\Delta \tilde u_n = \rho^2 \tilde u_n^p$ by $\tilde u_1$ and integrating by parts, we have that
\begin{equation} \label{lower1}
    \int_{\widetilde \Om} \rho^2 \,\tilde u_n^{p_n}\, \tilde u_1 = \int_{\widetilde \Om} \langle \nabla \tilde u_n, \nabla \tilde u_1 \rangle = \int_{\widetilde \Om} \tilde u_n (-\Delta \tilde u_1) = \lambda_1\int_{\widetilde \Om} \rho^2 \, \tilde u_n \, \tilde u_1. 
\end{equation}
Since $\tilde u_n^{p_n} \leq M_n^{p_n-1} \tilde u_n,$ then
\begin{equation} \label{lower2}
    \int_{\widetilde \Om} \rho^2 \, \tilde u_n^{p_n} \, \tilde u_1 \leq M_n^{p_n-1} \int_{\widetilde \Om} \rho^2 \, \tilde u_n \, \tilde u_1.
\end{equation}
Hence equations \eqref{lower1} and \eqref{lower2} imply that
\begin{equation}
    M_n^{p_n - 1} \geq \lambda_1.
\end{equation}

We now show that $M_n^{p_n - 1}$ is bounded from above for $p_n > 1$ sufficiently close to 1 by means of a blow-up argument. Suppose $M_n^{p_n - 1} \to +\infty$ as $p_n \downarrow 1,$ and set
\begin{equation*}
    u^*_n(x) := \frac{\tilde u_n(\varepsilon_n x + Q_n)}{M_n}, \quad \text{where} \; \; \varepsilon_n^2 = \frac{1}{M_n^{p_n - 1}}.
\end{equation*}

Since $\widetilde \Om$ is convex and $\rho^2$ is radially decreasing, by the method of moving planes (see \cite[Theorem 1']{GIDAS-NI-NIRENBERG}) there follows that $Q_n$ is away from the boundary $\partial \widetilde \Om.$ By standard elliptic estimates, $u^*_n \to u^* \in C^2(K)$ uniformly on each compact set $K \subset \R^2,$ and $u^*$ satisfies 
\begin{align*}
    \begin{cases}
    -\Delta u^* = \rho^2 u^*, \; u^*>0 \quad \text{in} \; \; \R^2, \\
    \;  u^*(0)  =1.
    \end{cases}
\end{align*}

Let $(\lambda_R, \tilde \varphi_R)$ be first Dirichlet eigenpair of the Laplacian $-\Delta$ in the Euclidean ball $B_R(0) \subset \R^2$ with weight $\rho^2,$ i.e. $\tilde \varphi_R$ is the solution of the eigenvalue problem 
    \begin{align} \label{BVP_eigen_ball}
        \begin{cases}
            -\Delta \tilde \varphi_R = \lambda_R \rho^2 \tilde \varphi_R \quad &\text{in} \; \; B_R(0) \subset \R^2, \\
            \qquad \tilde \varphi_R = 0 \quad &\text{on} \; \;  \partial B_R(0).
        \end{cases}
    \end{align}
Let us multiply the equation $-\Delta \tilde u = \rho^2 \tilde u^p$ by $\tilde \varphi_R,$ subtract it from the multiplication of equation \eqref{BVP_eigen_ball} by $u^*,$ and then integrate by parts. If $R$ is sufficiently large, then $\lambda_R <1.$ It follows that
\begin{align*}
    0 > \int_{\partial B_R (0)} u^* \rho^2 \frac{\partial \tilde \varphi_R}{\partial \nu}  = \int_{B_R (0)} u^* \Delta \tilde \varphi_R - \tilde \varphi_R \Delta u^*  = (1- \lambda_R) \int_{B_R(0)} u^* \rho^2 \tilde \varphi_R > 0,
\end{align*}
which is a contradiction.\\

\smallskip

\textit{Step 3. $M_n^{p_n - 1} \to \lambda_1.$} \\

Set $\overline{u}_n :=\displaystyle\frac{\tilde u_n}{M_n},$ which then satisfies the equation
\begin{align} \label{equation_ubar}
    \begin{cases}
        -\Delta \overline{u}_n = M_n^{p_n - 1} \rho^2 \, \overline{u}_n^{p_n} \quad &\text{in} \; \; \widetilde \Om, \\
        \quad \quad \, \overline{u} = 0 \; \; &\text{on} \; \partial \widetilde \Om.
    \end{cases}
\end{align}

By \textit{Step 2,} $M_n^{p_n-1}$ is uniformly bounded as $p_n \downarrow1.$ Hence up to a subsequence, we may assume $M_n^{p_n - 1} \to \lambda$ as $p_n \downarrow 1.$ Note that $\rho^2(x)$ is bounded for every $x \in \widetilde \Om,$ and also that $0 \leq \overline{u}_n \leq 1$ in $\Om$ by definition. Then for $p_n \in (1, 1+ \varepsilon),$ $\varepsilon>0,$ we have that $M_n^{p_n - 1} \rho^2 \, \overline{u}_n^{p_n}$ is uniformly bounded in $L^\infty(\Om).$ By standard elliptic estimates, $\overline{u}_n$ is uniformly bounded in $C^{2, \alpha}(\overline{\widetilde \Om}).$ Thence by Arzelà-Ascoli, up to a subsequence it follows that $\overline{u}_n \to \overline{u} \in C^2(\widetilde \Om) \cap C^0(\overline{\widetilde \Om})$ uniformly as $p_n \downarrow 1$. Passing to the limit in equation \eqref{equation_ubar}, we deduce that $\overline{u}$ satisfies 
\begin{align} \label{BVP_eigen2}
    \begin{cases}
        -\Delta \overline{u} = \lambda \rho^2 \overline{u}, \quad &\text{in} \; \, \Om, \\
        \;0 \leq \overline{u} \leq 1 \quad &\text{in} \; \, \Om, \\
        \quad  \; \; \, \overline{u} = 0 \quad &\text{on} \; \, \partial\Om.
    \end{cases}
\end{align}
Then $\overline{u} \geq 0,$ $\overline{u} \not \equiv 0$ gives that $\overline{u}>0$ by the strong maximum principle. By the variational characterization of the first eigenvalue, the eigenvalue problem \eqref{BVP_eigen2} admits a positive eigenfunction $\overline u$ if and only if $\lambda = \lambda_1.$ Moreover, the first eigenvalue $\lambda_1$ is simple, hence $\overline u$ must be a positive multiple of $\tilde u_1.$ Since $0 \leq \overline u \leq 1$ and $\tilde u_1$ is the first eigenfunction satisfying $\lVert \tilde u_1 \rVert_\infty =1$, we deduce that $\overline{u} =\tilde u_1.$\\
\smallskip

\textit{Step 4.} If  $\tilde u_n \not \equiv \tilde v_n$ are two solutions of \eqref{eq_stereo} with $p=p_n,$ $p_n \downarrow 1,$ then $\sup_{\widetilde \Om} \frac{\tilde u_n^{p_n} - \tilde v_n^{p_n}}{\tilde u_n - \tilde v_n} \to \lambda_1.$\\

Set $M_n := \displaystyle \sup_{\overline{\widetilde \Om}} \tilde u_n,$ $W_n := \displaystyle \sup_{\overline{\widetilde \Om}}\tilde v_n$ and $\overline u_n = \frac{\tilde u_n}{M_n},$ $\overline v_n = \frac{\tilde v_n}{W_n}.$ We proceed with the following argument for $\overline u_n$. By \textit{Step 3}, up to a subsequence, we have that $M_n^{p_n - 1} \to \lambda_1$ uniformly. Moreover, $\overline{u}_n = \frac{\tilde u_n}{M_n} \to \overline{u}$ uniformly in $\widetilde \Om,$ with $\overline u > 0$ on each compact $K \Subset \widetilde \Om.$ By writing $\overline{u}_n = \exp\big( (p_n - 1)\log \overline u_n \big),$ we note that $\overline u_n \geq c_K$ on $K$ for some constant $c_K > 0.$ Then $\log \overline u_n$ is uniformly bounded on $K$ and $(p_n - 1) \log \overline u_n \to 0$ uniformly on $K,$ and thus $\overline{u}_n^{p_n - 1} \to 1.$ Therefore $\tilde u_n^{p_n - 1} = M_n^{p_n - 1} \overline{u}_n^{p_n - 1} \to \lambda_1$ uniformly on each compact $K \Subset \widetilde \Om.$ By repeating the above argument for $\overline v_n$ we also get $\tilde v_n^{p_n - 1} \to \lambda_1$ uniformly on each compact $K \Subset \widetilde \Om.$ 

By the mean value theorem, for each $x \in \widetilde \Om$ we may write
\begin{equation}\label{Vn_integral}
V_n(x)=p_n \int_0^1\bigl(t\tilde u_n(x)+(1-t)\tilde v_n(x) \bigr)^{p_n-1}dt.
\end{equation}
For each $x \in \widetilde \Om,$ by $0 \leq \tilde u_n(x) \leq M_n$ and $0 \leq \tilde v_n(x) \leq W_n,$ we have that $0 \leq t\tilde u_n(x)+(1-t)\tilde v_n(x) \leq \max \{ M_n, W_n\}.$ The map $s \mapsto s^{p_n-1}$ is strictly increasing since $p_n > 1,$ hence $V_n(x) \leq p_n \max \{ M_n^{p_n - 1}, W_n^{p_n - 1}\}.$ Since $M_n^{p_n -1}\to \lambda_1$ and $W_n^{p_n -1} \to \lambda_1$ as $p_n \uparrow 1,$ by taking the supremum over $\widetilde \Om$ and passing to the limit,
\begin{equation} \label{limsupVn}
\limsup_{n \to \infty}  \sup_{\widetilde \Om} V_n \leq \lambda_1.
\end{equation}

On the other hand, let $Q_n \in \widetilde \Om$ be such that \eqref{aux_Qn} holds. At $Q_n,$ \eqref{Vn_integral} reads $V_n(Q_n) = p_n \int_0^1\bigl(t\tilde M_n+(1-t)\tilde v_n(Q_n) \bigr)^{p_n-1}dt.$
Since $\tilde v_n(Q_n) \geq 0,$ this last integrand is bounded from below by $t M_n,$ whilst again by the fact that the map $s \mapsto s^{p_n-1}$ is strictly increasing, we get $V_n(Q_n) \geq p_n M_n^{p_n - 1} \int_0^1 t^{p_n - 1} dt= M_n^{p_n - 1}.$ Thus
\begin{equation} \label{liminfVn}
\lambda_1 \leq \liminf_{n \to \infty}  \sup_{\widetilde \Om} V_n.
\end{equation}

Then by putting together the bounds \eqref{limsupVn} and \eqref{liminfVn}, one arrives at
\begin{equation} \label{supVn}
\sup_{\Om} V_n \to \lambda_1 \quad \text{as } p_n \uparrow 1.
\end{equation}\\

\smallskip

\textit{Step 5.} Conclusion. \\

To conclude the proof, suppose that the claim  is false and let $\tilde u_n \not \equiv \tilde v_n$ be two solutions of \eqref{eq_stereo} with $p=p_n,$ $p_n \downarrow 1.$  Let us define
 
\begin{equation*}
    \psi_n := \frac{\tilde u_n - \tilde v_n}{\lVert \tilde u_n-\tilde v_n \rVert_{\infty}},
\end{equation*}\\ \smallskip
which changes sign (by \textit{Step 1}) and satisfies
\begin{align*}
    \begin{cases}
        -\Delta \psi_n = V_n(x) \rho^2(x) \psi_n \quad &\text{in} \; \, \widetilde \Om \subset \R^2, \\
        \quad \; \, \, \psi_n = 0 \quad &\text{on} \; \, \partial \widetilde \Om.
    \end{cases}
\end{align*}

Let $\psi_n^{+}$ and $\psi_n^{-}$ denote respectively the positive and negative parts of $\psi_n,$ which are both nontrivial. In what follows, we write $\psi^{\pm}$ to indicate that the subsequent computations apply to each of these functions separately. By multiplying the equation above by $\psi_n^{\pm}$ and integrating by parts, one has
\begin{equation} \label{aux_Vn}
\int_{\widetilde \Om} |\nabla \psi_n^{\pm}|^2 = \int_{\widetilde \Om} V_n(x) \rho^2(x) (\psi_n^{\pm})^2.
\end{equation} 
By the characterization of the first eigenvalue $\lambda_1$, and by using \eqref{supVn} from \textit{Step 4} and equation \eqref{aux_Vn}, we have that
\begin{equation} \label{aux_wn}
\lambda_1 \leq \frac{\int_{\widetilde \Om} |\nabla \psi_n^{\pm}|^2}{ \int_{\widetilde \Om} \rho^2 (\psi_n^{\pm})^2} = \frac{  \int_{\widetilde \Om} V_n(x) \rho^2(x) (\psi_n^{\pm})^2}{ \int_{\widetilde \Om} \rho^2 (\psi_n^{\pm})^2} \leq \sup_{\widetilde \Om} V_n \to \lambda_1.
\end{equation}

Now set
\begin{equation*}
w_n^{\pm} := \displaystyle\frac{\psi_n^{\pm}\; \; \; \,}{ \lVert \psi_n^{\pm} \rVert_{L^2_{\rho^2}}},
\end{equation*}
where $\lVert f \rVert_{L^2_{\rho^2}} := \left( \int_{\widetilde \Om} \rho^2(x) f^2 \right)^{1/2}$ denotes the norm in the weighted Lebesgue space 
$$L^2_{\rho^2}(\widetilde \Om) := \{ f: \widetilde \Om \to \R \; \; \text{is a measurable function } \Big| \int_{\widetilde \Om} \rho^2(x) f^2  < \infty \}.$$
Note that the norms $\lVert \psi_n^{\pm} \rVert_{L^2_{\rho^2}}$ and $\lVert \psi_n^{\pm} \rVert_{L^2}$ are equivalent, as $0<\rho^2(x) \leq 4$ for every $x \in \widetilde \Om.$ Also, $w_n^{\pm}$ is well-defined since $\lVert \psi_n^{\pm} \rVert_{L^2_{\rho^2}} > 0.$ By equation \eqref{aux_wn} there holds $\int_{\widetilde \Om} |\nabla w_n^{\pm}|^2 \to \lambda_1,$ hence $(w_n^{\pm})$ is a bounded sequence in $H_0^1(\Om).$ Then up to subsequences,  
\begin{equation*}
w_n^{\pm} \rightharpoonup w^{\pm} \quad \text{weakly in } H_0^1(\widetilde \Om), \qquad
w_n^{\pm} \to w^{\pm} \quad \text{strongly in } L^2_{\rho^2}(\widetilde \Om).
\end{equation*}
Since $\lVert w_n^{\pm} \rVert_{L^2_{\rho^2}} = 1,$ then $\lVert w^{\pm} \rVert_{L^2_{\rho^2}} = 1.$ By the weak lower semicontinuity and the characterization of the first eigenvalue $\lambda_1$, one has
\begin{equation*}
\int_{\widetilde \Om} |\nabla w^{\pm}|^2 \leq \lim_{n \to \infty} \inf \int_{\widetilde \Om} |\nabla w_n^{\pm}|^2 = \lambda_1 = \lambda_1 \int_{\widetilde \Om} \rho^2 (w^{\pm})^2 \leq \int_{\widetilde \Om} |\nabla w^{\pm}|^2,
\end{equation*} 
thus equality holds and $w^\pm$ achieves the minimum in the Rayleigh quotient $\lambda_1= \inf_{\varphi \in H_0^1(\widetilde \Om \setminus \{0\})} \frac{\int_{\widetilde \Om} |\nabla \varphi|^2}{\rho^2(x) \varphi^2}.$ From $w_n^{\pm} \geq 0$ we have $w^\pm \geq 0,$ whilst by the simplicity of the first eigenvalue $\lambda_1$ and the strong maximum principle, we also have
\begin{equation*}
w_n^+ \to u_1 \quad \text{in } L^2_{\rho^2}(\widetilde \Om), \qquad 
w_n^- \to u_1 \quad \text{in } L^2_{\rho^2}(\widetilde \Om).
\end{equation*}
Furthermore, let us observe that $w_n^+ w_n^- = 0$ almost everywhere in $\widetilde \Om,$ whence 
\begin{equation} \label{aux_wn+-}
\int_{\widetilde \Om} \rho^2(x) w_n^+ w_n^- = 0 \quad \text{for every } n \in \N.
\end{equation} 
Finally, by recalling that $\rho^2 > 0,$ we pass to the limit in equation \eqref{aux_wn+-} to obtain
\begin{equation*}
0 = \int_{\widetilde \Om} \rho^2(x) u_1^2,
\end{equation*}
contradicting the positivity of the first eigenfunction $u_1.$ This concludes the proof. 
\end{proof} 

\begin{proof} [Proof of Corollary \ref{corollary_sub}]
It suffices to show that
\begin{equation}
\displaystyle \frac{\tilde u_n}{\lVert \tilde u_n \rVert_{\infty}} \to \tilde u_1  \quad \text{uniformly in } \widetilde \Om \subset \R^2,
\end{equation}
where $\tilde u_n := u_n \circ \pi^{-1},$ $\tilde u_1 :=  u_1 \circ \pi^{-1}$  and $\pi$ is the stereographic projection $\pi: \Om \to \widetilde \Om \subset \R^2.$ The argument closely parallels \textit{Steps 2} and \textit{3} in the proof of the superlinear case in Theorem \ref{theorem_uniqueness}. We indicate the relevant modifications. 

Let $M_n := \sup_{\overline {\widetilde \Om}} \tilde u_n = \tilde u_n(Q_n)$ for some $Q_n \in \widetilde \Om$ and let $(\lambda_1, \tilde u_1)$ be the first Dirichlet eigenpair of $-\Delta$ in $\Om$ with weight $\rho^2$ (see equation \eqref{BVP_eigen_R2}). Multiplying $-\Delta \tilde u_n = \tilde u_n^p$ by $\tilde u_1$ and integrating by parts, and then using that $\tilde u_n^{p_n} \geq M_n^{p_n - 1} \tilde u_n,$ we obtain $M_n^{1-p_n} \geq \frac{1}{\lambda_1}.$ Thus $M_n^{1-p_n}$ is bounded from below. Now suppose that $M_n^{1-p_n} \to +\infty$ and set $\overline u_n = \frac{\tilde u_n}{M^n},$ so that $0 \leq \overline u_n \leq 1$ and $\overline u_n$ satisfies equation \eqref{equation_ubar}. Note that $\lVert \overline u_n \rVert_\infty = 1,$ $\rho^2$ is bounded in $\widetilde \Om$ and $M_n^{p_n - 1} \to 0,$ hence $\lVert M_n^{p_n - 1} \rho^2 \overline u_n^{p_n} \rVert_\infty \to 0.$ By elliptic regularity, the sequence $(\overline u_n)_{n \in \N}$ is bounded in $C^{2, \alpha}(\overline{\widetilde \Om}),$ whence by Arzelà-Ascoli we have that $\overline u_n \to \overline u \in C^2(\widetilde \Om) \cap C^0 (\overline{\widetilde \Om}),$ up to a subsequence. Passing to the limit in equation \eqref{BVP_eigen_R2}, we thus obtain that $\overline u$ satisfies
\begin{align*}
    \begin{cases}
        -\Delta \overline u = 0 \quad &\text{in} \; \, \widetilde \Om, \\
        \quad \; \, \, \overline u = 0 \quad &\text{on} \; \, \partial \widetilde \Om.
    \end{cases}
\end{align*}
By the strong maximum principle, $\overline u \equiv 0.$ But $\lVert \overline u \rVert_\infty = 1,$ by the uniform convergence $\overline u_n \to \overline u.$ This gives a contradiction. Thus $M_n^{1-p_n}$ is also bounded from above. It is then straightforward to push through the same compactness argument done in \textit{Step 3} in the proof of Theorem \ref{theorem_uniqueness}, and the result follows.
\end{proof}

\begin{proof} [Proof of Corollary \ref{corollary_super}]
It follows directly from \textit{Steps 2} and \textit{Step 3} for the superlinear case in the proof Theorem \ref{theorem_uniqueness}.
\end{proof}

\section{Full rank of $\nabla^2 \left( u^{\frac{1-p}{2}} \right)$ near the boundary} \label{section_concavity_boundary}
This Section is devoted to the proof of Lemma \ref{lemma_boundary_rank} below, which proves that the conclusions \ref{sublinear} and \ref{superlinear} of Theorem \ref{power} hold provided one is sufficiently close to the boundary. This constitutes the ingredient \ref{ing3} needed for the proof of Theorem \eqref{power}, and sets the stage for the constant rank theorem in the next section.

In this section we use the following notation: for $\delta>0$ sufficiently small,
\begin{equation} \label{tubular_neighborhood}
\boxed{
\Om_{\delta}:= \{ x \in \Om \; \big| \; d(x, \partial \Om) < \delta \} \; \; \text{is a tubular neighborhood of $\partial \Om.$}
}
\end{equation} 

 We will work in Fermi coordinates near the boundary. Let $y(s)$ be a parametrization by arclength $s$ of $\partial \Om$ and let $\tau(s):= y'(s)$ be the unit tangent to $\partial \Om.$ We choose the orientation of $y(s)$ so that the geodesic curvature of $\partial \Om$ is positive with respect to the outward unit normal $\nu(s).$ Define 
\begin{align}
	X: \partial \Om \times [0, \delta) &\to \Om_\delta \\
	(s,d) &\mapsto \exp_{y(s)} \left( -\nu(s) d \right).
\end{align}
Here $d:= d(x, \partial \Om).$ The map $d \mapsto X(s,d)$ is the inward normal geodesic passing through $y(s) \in \partial \Om.$ We extend the unit tangent $\tau(s)$ and the unit inward normal $-\nu(s)$ along this geodesic via parallel transport, i.e. $-\nu(s,t):= X_d(s,d)$ and $\tau(s,d):= \frac{X(s,d)}{|X_s(s,d)|},$ where $|X_d(s,d)|= \cos d - \kappa_g(s) \sin d.$ For $\delta > 0$ sufficiently small, $X$ is a diffeomorphism onto $\Om_\delta,$ hence every $x \in \Om_\delta$ can be written as $x = X(s,d)$ for a unique boundary point $y(s)$ and a unique $d \in [0, \delta).$ At each $x \in \Om_\delta$ we have an orthonormal Fermi frame $\{\tau,-\nu \}:= \left\{ \frac{X(s,d)}{|X_s(s,d)|}, X_d(s,d) \right\}.$

\smallskip

\begin{lemma} \label{lemma_boundary_rank}
Let $\Om \subset \S^2$ be a uniformly convex domain with smooth boundary and let $u: \overline{\Om} \to \R$ be a solution of the Dirichlet problem \eqref{BVP0}. Then for a sufficiently small $\delta > 0,$ 
\begin{align}
v&=u^{\frac{1-p}{2}}  \text{ is strictly concave in $\Om_\delta$ for $0 \leq p < 1,$} \\
v&=u^{\frac{1-p}{2}} \text{ is strictly convex in $\Om_\delta$ for $p>1.$}
\end{align}
\end{lemma}

\smallskip

\begin{remark}
We emphasize that the condition $1<p\leq 3$ does not enter in this boundary argument. In fact, for $\delta>0$ sufficiently small and for every $p>1,$ the strict convexity of $u^{\frac{1-p}{2}}$ holds in $\Om_\delta,$  for every $p>1.$
\end{remark}

\begin{proof}[Proof of Lemma \ref{lemma_boundary_rank}]
Denote $\eta(s) := -\nu(s).$ By Hopf Lemma \cite{HOPF}, we have $u_\eta (y(s)) > 0$ for $y(s) \in \partial \Om.$ The map $y(s) \mapsto u_\eta\left( y(s) \right)$ is continuous and $\partial \Om$ is compact, hence 
 \begin{equation} \label{u_eta_boundary}
a_0:= \min_{y(s) \in \partial \Om} u_\eta\left(y(s) \right) > 0. 
\end{equation}

Let $u$ be a solution of \eqref{BVP0} for a fixed $p \in [0,1) \cup (1,\infty).$ Also, let $y(s) \in \partial \Om$ and, for $\delta>0$ sufficiently small, let $x=X(s,d) \in \Om_\delta$ denote the point at a distance $d:=d(x)$ from $y(s)$ and lying on the inward normal geodesic. Set $U(s,d):= u(X(s,d)).$ Since $u\big|_{\partial \Om} = 0,$ $X(s,0)=y(s)$ and $X_d(s,0) = \eta(s),$ then $U(s,0) = 0$ and $U_d(s,0) = u_\eta(y(s)).$ By Taylor's theorem, we have $U(s,d) = u_\eta(y(s)) d + O(d^2).$ We thus obtain
\begin{align}
u(x) &= u_{\eta}(y(s)) d + O(d^2) = O(d), \label{u_estimate} \\
u_\tau(x) &= \frac{1}{a(s,d)} U_s(s,d) = O(d), \label{ut_estimate} \\
u_\eta(x)&= u_\eta(y(s)) + O(d) = O(1). \label{un_estimate}
\end{align}
Moreover, after possibly taking $\delta>0$ smaller, \eqref{u_eta_boundary} yields
\begin{equation}
u_\eta(x) \geq \frac{a_0}{2} > 0.
\end{equation}

We now observe that for $y(s) \in \partial \Om$ we have $u_\tau (y(s)) = 0,$ while $0=\frac{d^2}{ds^2}u\left( y(s) \right) = \nabla^2 u(\eta, \eta) \left( y(s) \right) +  \langle \nabla u \left( y(s) \right), \nabla_\tau \tau \rangle = u_{\tau\tau}\left( y(s) \right)+\kappa_g(s) u_\eta \left( y(s)\right)$ gives that $u_{\tau \tau}(y(s)) = -\kappa_g(s) u_\eta (y(s)).$ Also, by the uniform convexity of $\Om$ we have that $\kappa_g(s)$ satisfies \eqref{uniform}, and $u \in C^\infty(\overline \Om)$ implies that $u_{\tau \tau}$ is continuous up to the boundary. Therefore, for $\delta > 0$ chosen sufficiently small, for every $x \in \Om_\delta$ we obtain 
\begin{align}
u_{\tau \tau}(x) &= -\kappa_g(s) u_\eta (y(s)) + O(d) = O(1), \label{u_tt_estimate} \\
u_{\tau \tau}(x) &\leq - \frac{\kappa_0 a_0}{2} < 0. \label{u_tt_lead}
\end{align}

For $x \in \Om_\delta,$ equation \eqref{BVP0} writes as $\Delta u(x) = u_{\eta \eta}(x) + u_{\tau \tau}(x) = - u^p(x)$ in the orthonormal frame $\{\eta, \tau \},$ and so for small enough $\delta>0,$ \eqref{u_estimate} and \eqref{u_tt_estimate} yield
\begin{align}
u_{\eta \eta}(x) &= \kappa_g(s) u_\eta(y(s)) + O(d) = O(1), \label{u_nn_estimate}\\
u_{\eta \eta}(x) &\geq \frac{\kappa_0 a_0}{2} > 0. \label{u_nn_lead}
\end{align}
Moreover, by differentiating $\langle \eta, \tau \rangle = 0$ in the tangential direction $\tau$ and using that $\langle \eta, \nabla_\tau \tau \rangle = \kappa_g,$ we get $\nabla_\tau \eta = -\kappa_g \tau.$ Since $u_\tau = 0$ on $\partial \Om,$ we have $\nabla u = u_\eta \eta$ on $\partial \Om.$ Hence $\langle \nabla u,\nabla_\tau\eta \rangle = \langle u_\eta\eta,-\kappa_g \tau \rangle=0.$ This yields $u_{\eta\tau}\left(y(s)\right) = \frac{d}{ds}\bigl(u_\eta(y(s))\bigr) -\langle\nabla u,\nabla_\tau\eta\rangle = \frac{d}{ds}\bigl(u_\eta(y(s))\bigr).$ By recalling that $u$ is smooth in $\overline \Om,$ that $y(s)$ is a smooth parametrization of $\partial \Om$ and that $\eta$ is smooth along $\partial \Om,$ we have that $u_\eta\left( y(s) \right)$ is smooth on $\partial \Om$ (which is compact). Thus  $\| \frac{d}{ds} \bigl(u_\eta(y(s))\bigr) \| \leq C$ for some $C>0,$ and by Taylor's theorem we obtain that
\begin{equation}
u_{\eta \tau}(x) = u_{\eta \tau}\left(y(s)\right) + O(d) = O(1).
\end{equation}

Now let $v:=u^\alpha,$ $\alpha = \displaystyle \frac{1-p}{2}.$ Then 
\begin{equation}
\nabla^2 v = \alpha u^{\alpha - 1} \nabla^2 u + \alpha (\alpha-1) u^{\alpha -2} du \otimes du,
\end{equation}  
and we estimate the components of $\nabla^2 v$ in the orthonormal frame $\{\tau, \eta\}$ at $x \in \Om_\delta$ below. We use the notation $h(x) \asymp k(x)$ to say that there exist constants $c, C > 0$ such that $c k(x) \leq h(x) \leq C k(x).$ In the next, $c,C>0$ depend only on $\alpha,$ $\delta>0,$ and on the lower bound $\kappa_0$ on the geodesic curvature of $\partial \Om.$

• $ 0 \leq p < 1$\\
By the estimates for $u(x)$ with $x \in \Om_\delta$ for some $\delta>0$ sufficiently small, we have
\begin{align}
v_{\eta \eta}(x) &\asymp \alpha(\alpha-1) d^{\alpha -2} < 0, \quad v_{\eta \eta}(x) = O(d^{\alpha -2}), \\
v_{\tau \tau}(x) &\asymp \alpha(\alpha-1) d^{\alpha -1} < 0, \quad v_{\tau \tau}(x) = O(d^{\alpha -1}), \\
v_{\tau \eta}(x) &= O(d^{\alpha - 1}).
\end{align}

Thus there exists $\delta>0$ such that, for every $x \in \Om_\delta,$ there holds $v_{\eta \eta}(x) < 0$ and $v_{\eta \eta}(x) v_{\tau \tau}(x) - v_{\eta \tau}^2(x) > 0.$ Thus $v$ is strictly concave in $\Om_\delta,$ by Sylvester's criterion.

• $ p > 1$\\
Again by the estimates for $u(x)$ with $x \in \Om_\delta$ for some $\delta>0$ sufficiently small, 
\begin{align}
v_{\eta \eta}(x) &\asymp \alpha(\alpha-1) d^{\alpha -2} > 0, \quad v_{\eta \eta}(x) = O(d^{\alpha -2}), \\
v_{\tau \tau}(x) &\asymp \alpha(\alpha-1) d^{\alpha -1} > 0, \quad v_{\tau \tau}(x) = O(d^{\alpha -1}), \\
v_{\tau \eta}(x) &= O(d^{\alpha - 1}).
\end{align}
Then there exists $\delta>0$ such that, for every $x \in \Om_\delta,$ we have $v_{\eta \eta}(x) > 0$ and $v_{\eta \eta}(x) v_{\tau \tau}(x) - v_{\eta \tau}^2(x) > 0.$ Then $v$ is strictly convex in $\Om_\delta.$
\end{proof}

\section{A Constant Rank Hessian result} \label{section_constant_rank}

In this Section we establish ingredient \ref{ing4} needed to prove Theorem \ref{power}: a constant rank theorem for the covariant Hessian of $u^{\frac{1-p}{2}},$ where $u$ is a solution to problem \eqref{BVP0} with $p \in [0,1) \cup (1,3].$ Our proof was inspired by (Theorem 1.1, \cite{CAFFARELLI-FRIEDMAN}) and (Theorem 1, \cite{KOREVAAR-LEWIS}). We recall that $v=u^{\frac{1-p}{2}}$ satisfies 
\begin{align}
\boxed{
\Delta v = \frac{1}{v} \left( -\frac{1+p}{1-p} |\nabla v|^2 - \frac{1-p}{2}\right) \quad \text{in } \Om \subset \S^2.
}  \label{deltav_sub2}
\end{align}

\begin{theorem} \label{MAIN}
Let $\Om  \subset \S^2$ be a domain with smooth boundary and let $u \in C^\infty(\overline \Om)$ be a solution of problem \eqref{BVP0}. Let $v: \overline \Om \to \R$ be given by $v:= u^{\frac{1-p}{2}},$ and denote by $\nabla^2 v$ its covariant Hessian. Suppose that,
\begin{enumerate}[label=(\alph*)]
\item for $0\leq p <1,$ $\nabla^2 v$ is negative semidefinite in $\Om;$ \label{negsemi_sub}
\item for $1 < p \leq 3,$ $\nabla^2 v$ is positive semidefinite in $\Om.$ \label{possemi_super}
\end{enumerate}
Then $\nabla^2 v$ has constant rank in $\Om.$
\end{theorem} 

\begin{corollary} \label{corollary_constant_rank}
Let $\delta>0$ be sufficiently small and let $\Om_\delta := \{ x \in \Om : d(x, \partial \Om) < \delta \}.$ Under the hypotheses of Theorem \ref{MAIN}, suppose further that,
\begin{enumerate}[(i)]
\item for $0 \leq p < 1,$ $\nabla^2 v$ is negative definite in $\Om_\delta;$ \label{sub_constantrank}
\item for $1 < p \leq 3,$ $\nabla^2 v$ is positive definite in $\Om_\delta.$ \label{super_constantrank}
\end{enumerate}
Then \ref{sub_constantrank} $\nabla^2 v$ is negative definite in $\Om;$  \ref{super_constantrank} $\nabla^2 v$ is positive definite in $\Om.$\\
\end{corollary}

\begin{remark}
Theorem \ref{MAIN} may be extended to monotone transformations $v=g(u)$ of positive solutions $u$ to a Dirichlet problem $\Delta u=f(u)$ in $\Om \subset \S^2,$ $u\big|_{\partial \Om}=0,$ for which $\Delta v = M(v)(a|\nabla v|^2 + b),$ where $M \in C^2(I),$ $M>0$ on $I \supset v(\Om ),$ and $a,b \in \R.$ Then $\nabla^2 v$ has constant rank in $\Om$ provided we assume that either
\begin{enumerate}
\item $\nabla^2 v$ is positive semidefinite in $\Om;$ $a,b >0;$ $\frac{1}{M}$ is convex in $I;$ and $2\frac{M'}{M} + aM \geq 0$ in $I.$
\item  $\nabla^2 v$ is negative semidefinite in $\Om;$ $a,b < 0,$ $\frac{1}{M}$ is concave in $I;$ and $2\frac{M'}{M} + aM \leq 0$ in $I.$
\end{enumerate}
Note Theorem \ref{MAIN} is recovered by taking $I=(0,+\infty),$ $M(s)=\frac{1}{s}>0$ in $I,$ $a=-\frac{1+p}{1-p}$ and $b=-\frac{1-p}{2}.$ The cases $0\leq p <1$ and $1<p\leq 3$ correspond respectively to the sets of conditions \textit{(1)} and \textit{(2)} above. The first Dirichlet eigenfunction $u_1$ also fits this framework: for $v=\log u_1,$ one has $I=(-\infty,0),$ $M \equiv 1,$ $a=-1,$ $b=-\lambda_1,$ hence the set \textit{(2)} of hypotheses hold.
\end{remark}

\smallskip

Since our argument is local, at each point $x \in \Om$ we may consider spherical coordinates $(\theta, \phi) \in (0, \pi) \times (0, 2 \pi)$ on a neighborhood $U$ of $x$ such that $U$ is away from the poles. In these coordinates, $g= d\theta^2 + \sin^2 \theta d\phi^2$ and $\{e_1, e_2\}= \{ \partial_\theta, \frac{\partial_\phi}{\sin \theta} \}$ is a well-defined smooth local orthonormal frame in $U.$ For each $x \in \Om,$ we denote by $H(x)$ the covariant Hessian matrix $(\nabla^2v)_{ij}(x)$ in this local orthonormal frame:  
\begin{equation} \label{hessian_matrix}
H(x) :=
\begin{pmatrix}
A & B \\
B & C
\end{pmatrix},
\end{equation}
with 
\begin{align}
\qquad \qquad \quad A &:= v_{\theta \theta} \big|_x, \\
\qquad \qquad \quad B &:= \displaystyle\frac{1}{\sin \theta} \left(v_{\theta \phi}\big|_x - \cot \theta v_\phi\big|_x \right), \\
\qquad \qquad \quad C &:=\displaystyle\frac{v_{\phi \phi}\big|_x}{\sin^2 \theta} + \cot \theta v_\theta\big|_x ,
\end{align}
where subscripts denote partial derivatives, e.g. $v_{\theta \theta} = \partial^2_{\theta \theta} v.$ \\

\begin{proof} [Proof of Theorem \ref{MAIN}]
Define the following auxiliary function:
\begin{align}
        \varphi(x) := \det H(x) = AC - B^2. \label{def_varphi}
\end{align} 

• $1 < p \leq 3$ \\
By hypothesis \ref{possemi_super}, $H(x)$ is positive semidefinite for every $x \in \Om$. From \eqref{deltav_sub2}, we see that  $\Delta v > 0$ in $\Om.$ Then for each $x \in \Om,$ $H(x)$ has eigenvalues $\mu_1, \mu_2$ that satisfy
\begin{align}
     \mu_1 + \mu_2 = \Delta v\big|_x &> 0, \label{S1}\\
     \mu_1  \mu_2 = \varphi\big|_x &\geq 0\label{S2}.
\end{align}

 By \eqref{S1}, it must be that $rank \, H(x) \geq 1$ $\forall x \in \Om.$ If there exists no point $x_0 \in \Om$ such that $rank \, H(x_0) = 1,$ then $rank \, H(x) = 2$ for every $x \in \Om,$ and there is nothing to prove. Let us then assume that $x_0 \in \Om$ is a point for which $rank \, H(x_0) =1,$ which implies that $\varphi(x_0)=0.$ We will prove that
\begin{equation}\label{varphi_0}
\varphi \equiv 0 \quad \text{in } \Om.\\
\end{equation}
Assume that $\mu_1 \neq 0$ and $\mu_2 = 0$ at the point $x_0,$ and pick some constant $\varepsilon>0$ so that $\mu_1$ is bounded from below by $2 \varepsilon$ at $x_0.$ Let then $U \Subset \Om$ be a small neighborhood of $x_0$ such that $\mu_1$ is bounded from below by $\varepsilon,$ and let $x \neq x_0$ be an arbitrary point in $U.$  We may also assume (up to an isometry) that $x$ lies on the equator $\theta = \pi/2$ and that the local orthonormal frame $\{\partial_\theta, \frac{\partial_\phi}{\sin \theta}\},$ which is well-defined in $U,$ diagonalizes $H(x).$ Thence at $x \in U$ we have
\begin{align}
    A &= v_{\theta \theta} \big|_x \geq \varepsilon, \label{v11_bound}\\
    B &= v_{\theta \phi}\big|_x = 0. \label{v21_zero} 
\end{align}

We proceed by following the next three steps. \\

\textit{Step 1.} $\Delta \varphi \lesssim 0 \quad  \text{in } U.$
\medskip

As in \cite{CAFFARELLI-FRIEDMAN}, if $h$ and $k$ are two functions defined in an open set $U \subset \Om,$ for any $x \in U$ we say that $h(x) \lesssim k(x)$ if there exist positive constants $c_1, c_2 > 0$ such that
\begin{equation} \label{relation_sim}
    (h - k)(x) \leq (c_1 |\nabla \varphi| + c_2 \varphi) (x),
\end{equation}
and we say that $h(x) \sim k(x)$ if both $h(x) \lesssim k(x)$ and $k(x) \lesssim h(x).$ We shall also write $h \lesssim k$ if the above inequality holds in $U,$ with the constants $c_1,c_2$ being independent of $x \in U.$ Moreover, we write $h \sim k$ if both $h \lesssim k$ and $k \lesssim h.$ 

In order to prove $\Delta \varphi \lesssim 0$ in $U$, we start by giving the following relations, which hold at the point $x \in U.$ From equation \eqref{v21_zero} we obtain 
 \begin{align}
     0 \sim \varphi(x) = AC-B^2 \sim AC \implies C \sim 0 \implies v_{\phi\phi}\big|_x \sim 0.\label{sim_phi}
 \end{align}
By \eqref{S2} and our hypothesis $\varphi(x_0)=0,$ we have that $\varphi_\theta(x) \sim 0$ and $\varphi_\phi(x) \sim 0.$ Then by differentiating \eqref{def_varphi} and then using \eqref{sim_phi}, we also obtain
\begin{align}
    \qquad &0 \sim \varphi_\theta(x) = A_\theta C + A C_\theta - 2 B B_\theta \; \, \sim A C_\theta \implies C_\theta \sim 0 \implies v_{\theta \phi \phi}\big|_x \sim v_\theta\big|_x , \label{sim_phi_t}  \\
     \qquad &0 \sim \varphi_\phi(x) = A_\phi C + A C_\phi - 2 B B_\phi \sim A C_\phi \implies C_\phi \sim 0 \implies v_{\phi\phi\phi}\big|_x \sim 0. \label{sim_phi_p}
\end{align}
In turn, from the previous relations we obtain the following at $x \in U:$ 
\begin{align}
\Delta \varphi (x) &= A \Delta C + C \Delta A + 2 \langle \nabla A, \nabla C \rangle - 2 |\nabla B|^2 - 2B \Delta B \\
&\sim A \Delta C - 2 |\nabla B|^2 \\
&=  A(C_{\theta \theta} + C_{\phi \phi}) - 2(B_\theta^2 +  B_\phi^2), \label{Deltaphi_1}
\end{align}
since $x$ lies on the equator $\theta=\pi/2$ and so $\sin \theta = 1,$ $\cos \theta = 0$ there. Moreover, we notice that from \eqref{sim_phi}--\eqref{sim_phi_p} we obtain the following at the point $x \in U:$
\begin{align} \label{BC_relations}
\qquad &C_{\theta \theta} = v_{\theta \theta \phi \phi}\big|_x + 2 v_{\phi \phi}\big|_x - 2 v_{\theta \theta}\big|_x, \qquad \; \; \;C_{\phi \phi} = v_{\phi \phi \phi \phi}\big|_x, \\
\qquad &B_{\theta} = v_{\theta \theta \phi}\big|_x + v_{\phi}\big|_x, \qquad \qquad  \qquad \qquad \;\,B_\phi = v_{\theta \phi \phi}\big|_x. 
\end{align}
We now substitute the relations \eqref{BC_relations} back into equation \eqref{Deltaphi_1}. Note that from \eqref{sim_phi}, at $x$ we have that $\Delta v (x) = v_{\phi \phi}\big|_x + v_{\theta \theta}\big|_x \sim v_{\theta\theta}\big|_x.$ We then have $(\Delta v)_{\phi}(x) = v_{\theta \theta \phi} + v_{\phi\phi\phi}\big|_x \sim v_{\theta \theta \phi}\big|_x$ (by \eqref{sim_phi_p}), and $(\Delta v)_{\phi \phi}(x) \sim v_{\theta \theta \phi \phi}(x) +  v_{\phi \phi \phi \phi}(x).$ Thus
\begin{align}
\Delta \varphi(x) &\sim v_{\theta \theta}  \Bigl[ v_{\theta \theta \phi \phi} + v_{\phi \phi\phi\phi}  +2 v_{\phi \phi} - 2 v_{\theta \theta} \Bigr] - 2 (v_{\theta \theta \phi} + v_\phi)^2 - 2 (v_{\theta \phi \phi} )^2 \\
&\sim v_{\theta \theta}   \Bigl[ (\Delta v)_{\phi \phi}    - 2 \Delta v    + 4 v_{\phi \phi}   \Bigr] - 2 \bigl( (\Delta v)_\phi  + v_\phi \bigr)^2 - 2 (v_{\theta \phi \phi} )^2. \label{Deltaphi_estimate_0}
\end{align} 
In addition, from \eqref{deltav_sub2} we have that $\Delta v = \frac{1}{v} \left( -\frac{1+p}{1-p} |\nabla v|^2 - \frac{1-p}{2}\right)$ at $x \in U$, from which we estimate $(\Delta v)_\phi,$ $(\Delta v)_\phi$ in as follows (here we shall denote $a:= -\frac{1+p}{1-p}$ and $b:=- \frac{1-p}{2}$ in order to simplify our computations). From the relations \eqref{v21_zero}, \eqref{sim_phi} and the fact that $x$ lies on the equator $\theta=\pi/2,$ at $x \in U$ there holds 
\begin{align} 
(\Delta v)_\phi &= 2a \frac{1}{v} v_\theta v_{\theta\phi} + 2a \frac{1}{v} v_\phi v_{\phi\phi} - \frac{1}{v^2}\bigl(a|\nabla v|^2+b\bigr)v_\phi \\
&\sim -\frac{1}{v^2}\bigl(a|\nabla v|^2+b\bigr)v_\phi \\
&= -\frac{1}{v} \Delta v \, v_\phi. \label{Deltav2}
\end{align}
Moreover, by the relations \eqref{v21_zero} and \eqref{sim_phi}--\eqref{sim_phi_p}, we obtain
\begin{align}
	 (\Delta v)_{\phi \phi} &= \frac{2a}{v} v_{\theta\phi}^2 +\frac{2aM}{\sin^2\theta}v_{\phi\phi}^2 + \frac{2a}{v} v_\theta v_{\theta\phi\phi} + \frac{2a}{v} \frac{1}{\sin^2\theta}v_\phi v_{\phi\phi\phi} -\frac{4a}{v^2}  v_\theta v_{\theta\phi}v_\phi \\
	 &\quad-\frac{4a}{v^2}\frac{1}{\sin^2\theta}v_\phi^2v_{\phi\phi} - \frac{1}{v^2}(a|\nabla v|^2+b)v_{\phi\phi} + \frac{2}{v^3}(a|\nabla v|^2+b)v_\phi^2 . \\
	 &\sim \frac{2}{v^3} \bigl(a|\nabla v|^2+b\bigr)v_\phi^2 + \frac{2a}{v} v_\theta^2.\\
	 &= \frac{2}{v^2} \Delta v \, v_\phi^2 + \frac{2a}{v} v_\theta^2. \label{Deltav22}
    \end{align}
After substituting \eqref{Deltav2} and \eqref{Deltav22} back into \eqref{Deltaphi_estimate_0}, we may estimate $\Delta \varphi(x)$ as follows. Below we use that $\Delta v(x) \sim v_{\theta\theta}\big|_x,$ $|\nabla v|^2\big|_x = v_\theta^2\big|_x + v_\phi^2\big|_x$ (since $\sin \theta=1$ at $x$) and the relation  \eqref{sim_phi_t}. Then at $x \in U,$ 
{\allowdisplaybreaks
\begin{align}
\Delta \varphi(x) &\sim v_{\theta \theta}   \Bigl[ (\Delta v)_{\phi \phi}    - 2 \Delta v    + 4 v_{\phi \phi}   \Bigr] - 2 \bigl( (\Delta v)_\phi  + v_\phi \bigr)^2 \frac{\Delta v}{\Delta v} - 2 v_{\theta}  ^2 \\
&\sim  v_{\theta \theta}  \Biggl[ (\Delta v)_{\phi \phi}  - 2 \frac{\bigl( (\Delta v)_\phi + v_\phi  \bigr)^2}{\Delta v}  - 2 \Delta v   \Biggr]- 2 v_{\theta}^2 \frac{\Delta v}{\Delta v}, \\
&\sim v_{\theta \theta} \Biggl[ \frac{2}{v^2} \Delta v \, v_{\phi}^2 + \frac{2a}{v} v_{\theta}^2 - \frac{2}{\Delta v} \left(-\frac{\Delta v}{v} + 1 \right)^2 v_\phi^2 - \frac{2a}{v} |\nabla v|^2 - \frac{2b}{v} \Biggr] - 2 v_{\theta}^2 \frac{v_{\theta\theta}}{\Delta v} \\
&\sim v_{\theta \theta} \Biggl[ \frac{2}{v^2} \Delta v \, v_{\phi}^2 - \frac{2a}{v} v_{\phi}^2 - \frac{2}{\Delta v} \left(\frac{(\Delta v)^2}{v^2} - 2 \frac{\Delta v}{v} + 1\right) v_\phi^2 - \frac{2b}{v} - \frac{2}{\Delta v}v_\theta^2 \Biggr] \\
&= v_{\theta \theta} \Biggl[ \frac{2}{v} (2-a) v_{\phi}^2 - \frac{2b}{v} - 2\frac{|\nabla v|^2}{\Delta v}v_\theta^2 \Biggr] \\
&\lesssim  v_{\theta \theta} \Biggl[ \frac{2}{v} (2-a) v_{\phi}^2 - \frac{2b}{v} \Biggr], \label{Deltaphi_estimate}
\end{align}
}
where we have discarded the term $- 2\frac{|\nabla v|^2}{\Delta v}v_\theta^2 $ since it is nonpositive. 

Finally, after substituting back for $a:= -\frac{1+p}{1-p}$ and $b:=- \frac{1-p}{2}$ in \eqref{Deltaphi_estimate}, we then use \eqref{v11_bound}, $v>0$ (since $u>0$), and the assumptions $\Delta v > 0$ and $1 <p \leq 3,$ to obtain that there exists some constant $c>0$ such that
\begin{align}
\Delta \varphi(x) &\lesssim v_{\theta \theta} \Biggl[ \frac{2}{v} \left( \frac{3-p}{1-p} \right) v_{\phi}^2 + \frac{1-p}{v}  \Biggr]  \leq - c \varepsilon  < 0 \quad \text{at } x \in U.
\label{Deltaphi_finalestimate}
\end{align}

Since we had chosen an arbitrary $x \neq x_0$ in the neighborhood $U \ni x_0,$ this proves $\Delta \varphi \lesssim 0$ in $U$. It is important to stress that for $p>3,$ we cannot ensure that \eqref{Deltaphi_finalestimate} holds; this is the only place where the hypothesis $1\leq p< 3$ enters the argument. \\

\textit{Step 2.} $\varphi \equiv 0$ in $U.$
\smallskip

By \textit{Step 1} and definition \eqref{relation_sim},  $\exists\, c_1, c_2>0,$ independent of $x,$ such that 
\begin{equation}
\Delta \varphi - c_1 |\nabla \varphi| - c_2 \varphi \leq 0 \quad \text{in } U. \label{Deltaphi_final}
\end{equation}
\smallskip
Moreover, observe that $\varphi$ is smooth in $U,$ whilst $\varphi \geq 0$ in $\overline U$ (by \eqref{S2}) and $\varphi(x_0) = 0$ (by our hypothesis that $H(x_0)$ has rank 1). By the strong minimum principle, we deduce that $\varphi \equiv 0$ in $U.$ \\

\textit{Step 3.} $\varphi \equiv 0$ in $\Om.$
\smallskip

For this we follow the argument in (Theorem 1, \cite{KOREVAAR-LEWIS}) and claim that
\begin{equation}
E = \{ x \in \Om \; \big| \; \rank H(x) = 1 \} \neq \emptyset  \quad \text{is open and closed in } \Om. \label{E_connect} 
\end{equation}
\smallskip
It is clear that $E$ is non-empty since $x_0 \in E$ (by hypothesis). Also, $E$ is open in $\Om$ since $\varphi \equiv 0$ in the neighborhood $U \Subset \Om$ (by our last argument). On the other hand, $E$ is closed in $\Om$ because the continuity of $\varphi$ implies that it must be that $\rank H(y_0) \leq 1$ for any boundary point $y_0 \in \partial E,$ and so we may repeat the above argument (with $y_0$ instead of $x_0$) to obtain that actually $\rank H = 1$ in a neighborhood of $y_0.$ The connectedness of $\Om \subset \S^2$ then implies that $E = \Om.$ This proves the claim \eqref{E_connect}, which proves \eqref{varphi_0}. Therefore, $rank \, H(x) = 1$ for every point $x \in \Om.$ 

\bigskip

• $0 \leq p <1.$\\
 By hypothesis \ref{negsemi_sub}, $H(x)$ is negative semidefinite $\forall x \in \Om.$ From equation \eqref{deltav_sub2}, $\Delta v < 0.$ Then for each $x \in \Om,$ $H(x)$ has eigenvalues $\mu_1, \mu_2$ satisfying 
\begin{align}
     \mu_1 + \mu_2 = \Delta v\big|_x &< 0, \label{S1_new}\\
     \mu_1  \mu_2 = \varphi\big|_x &\leq 0\label{S2_new}.
\end{align}
As in the proof of the case \ref{possemi_super}, we assume there exists $x_0 \in \Om$ such that $\rank H(x_0) = 1,$ which implies $\varphi(x_0)=0$ (or else the proof is done). Again we wish to prove that \eqref{varphi_0} holds. We do this by pushing through the steps of the previous case almost verbatim. The relevant changes are as follows. In the point $x \in U,$ where $U$ is a neighborhood of $x_0,$ we have $A= v_{\theta\theta} \big|_x \leq - \varepsilon < 0$ for some $\varepsilon >0$ (instead of \eqref{v11_bound}). Since $v>0$ and $0 \leq p <1,$ this yields the estimate
\begin{align}
\Delta \varphi(x) &\lesssim v_{\theta \theta} \Biggl[ \frac{2}{v} \left( \frac{3-p}{1-p} \right) v_{\phi}^2 + \frac{1-p}{v}  \Biggr]  \leq - \hat c \varepsilon  < 0 \quad \text{at } x \in U,
\label{Deltaphi_finalestimate2}
\end{align}
and so $\Delta \varphi \lesssim 0$ in $U.$  By setting $\psi:= -\varphi,$ this implies $\Delta \psi \gtrsim 0$ in $U.$ By \eqref{relation_sim}, this means there exist constants $\hat c_1, \hat c_2 >0$ such that
\begin{equation}
\Delta \psi - \hat{c_1} |\nabla \psi| - \hat{c_2} \psi \geq 0 \quad \text{in } U.
\end{equation}
By \eqref{S2_new}, note that $\psi\big|_x \geq 0.$ Also, $\psi(x_0)=0$ by our hypothesis. We may thus apply the strong maximum principle to obtain $\psi \equiv 0$ in $U,$ and hence $\varphi \equiv 0$ in $U$ as well. Thus the claim \eqref{E_connect} again follows. The rest of the argument goes unchanged. Therefore $\rank \,  H(x)=1$ for each $x \in \Om,$ which completes our proof.
\end{proof}

\medskip

\begin{proof}[Proof of Corollary \ref{corollary_constant_rank}]
Let $0 \leq p <1.$ By hypothesis \ref{sub_constantrank}, $\nabla^2 v$ is negative definite in $\Om_\delta,$ for some sufficiently small $\delta>0.$ By Theorem \ref{MAIN}, $\nabla^2 v$ has constant rank in $\Om.$ Then $\nabla^2 v$ is negative definite in $\Om.$ The case $1 < p \leq 3$ follows analogously.
\end{proof}

\bigskip

\section{Proof of Theorem \ref{power}} \label{section_proof_power}
We are finally in a position to put together the 4 ingredients needed to prove Theorem \ref{power}. We consider a uniformly convex domain $\Om \subset \S^2$ and use a continuation argument on the exponent $p$ to prove the cases \ref{sublinear} and \ref{superlinear}, recalling that the case \ref{eigenvalue} follows from \cite{LEE-WANG} (see Proposition \ref{log_eigen}). 
 
\begin{proof} [Proof of Theorem \ref{power}]
We start the proof by claiming that
\begin{align}
&v:= u_p^{\frac{1-p}{2}} \text{ is strictly concave for } p \in (1-\varepsilon, 1); \label{claim_sub} \\
&v:= u_p^{\frac{1-p}{2}} \text{ is strictly convex for }   p \in (1, 1+ \varepsilon). \label{claim_super}
\end{align}
From the case \ref{eigenvalue} in Theorem \ref{power}, $w:= -\log u_1$ is strictly convex in $\Om.$ By Corollary \ref{corollary_sub}, as $p=p_n \uparrow 1$ we have that $u_p/M_p \to u_1$ uniformly. Similarly, by Corollary \ref{corollary_super} we also have $u_p/M_p \to u_1$ uniformly as $p \downarrow 1.$  Thence $-\log u_p$ is strictly convex for each $p \in (1-\varepsilon, 1+\varepsilon),$  with $\varepsilon > 0$ sufficiently small. 
Let us now fix some  $1-\varepsilon < p<1.$ Set $w:= -\log u_p,$ $\alpha:=\frac{1-p}{2}>0,$ and write $v=e^{-\alpha w}.$ Therefore
$ \nabla^2 v = e^{-\alpha w}\bigl(\alpha^2\,\nabla w\otimes \nabla w-\alpha\,\nabla^2 w\bigr)$ and
\begin{equation} \label{aux_zeta}
\zeta^{T}\nabla^2 v\,\zeta = e^{-\alpha w}\Bigl(\alpha^2\langle \nabla w,\zeta\rangle^2-\alpha\,\zeta^T \nabla^2 w\,\zeta\Bigr) \quad \forall \,  \zeta\in T_x\Om .
\end{equation}
Since $w$ is strictly convex, we have $\zeta^{T}\nabla^2 w\,\zeta>0$ for every $\zeta\neq0$. Hence, since $1-\varepsilon < p < 1,$ then $\alpha>0$ is sufficiently small, the negative term in \eqref{aux_zeta} dominates and $\zeta^{T}\nabla^2 v\,\zeta<0$ for every $\zeta\neq0.$ This shows \eqref{claim_sub}. 

To prove the claim \eqref{claim_super}, we now fix some $1 < p < 1+\varepsilon$ and proceed as follows. By Theorem \ref{theorem_uniqueness}, there exists some exponent $p_0>1$ such that, for $1<p<p_0,$ problem \eqref{BVP0} has a unique positive solution $u_p.$ For this range, $u_p$ is thus the least-energy solution. As before, set $w:= -\log u_p,$ $\alpha:= \frac{1-p}{2},$ and write $v=e^{-\alpha w},$ so that we again obtain \eqref{aux_zeta}. Here we consider $\varepsilon < p_0 - 1.$ Then, since $1 < p < 1+\varepsilon,$ then $\alpha^2 > 0$ is sufficiently small, and the term $-\alpha \zeta^T \nabla^2 w \zeta > 0$ dominates. Thus $\zeta^T \nabla^2 v \zeta > 0$ for every $\zeta \neq 0,$ and this shows \eqref{claim_super}.

So far we have shown that the conclusion of Theorem \ref{power} holds in a small interval centered at $p=1.$ If the conclusion fails, then this must happen for some $0 \leq p_* \leq 1-\varepsilon$ (sublinear case) or for some $1+\varepsilon \leq p^* \leq 3$ (superlinear case). We thus argue by contradiction and analyze the following two cases: \\ 

\textit{• The sublinear case, $0\leq p<1$} \hfill\\ 
Let us define the set $S := \{ p \in [0, 1) \,  \big| \,  u^{\frac{1-p}{2}} \text{ is not strictly concave in } \Om \}$ and set $p_* := \sup S$ and $v_*:= \displaystyle u_{p^*}^{\frac{1-p_*}{2}}.$ Clearly $0 \leq p_* \leq 1- \varepsilon$ and $\nabla^2 v_*$ is negative semidefinite in $\Om.$ Then there exists $x_0 \in \Om$ such that $rank\bigl( \nabla^2 v_*(x_0) \bigr) = 1.$ Moreover, we see that $v_*$ satisfies equation \eqref{deltav_sub2}, hence we are under the hypotheses of Theorem \ref{MAIN}. By Lemma \ref{lemma_boundary_rank}, $\nabla^2 v_*$ is negative definite in a tubular neighborhood of the boundary, $\Om_\delta =\{ x \in \Om : d(x, \partial \Om) < \delta \},$ for some $\delta > 0.$ Then by Corollary \ref{corollary_constant_rank}, $\nabla^2 v_*$ must be positive definite in $\Om.$ This gives a contradiction. It then follows that $S = \emptyset,$ and so $u^{\frac{1-p}{2}}$ is strictly concave in $\Om$ for each $0 \leq p < 1.$\\
\medskip

\textit{• The superlinear case, $1 < p\leq3$} \hfill\\ 
Let us first consider the case $1+\varepsilon < p < p_0,$ which by Theorem \ref{theorem_uniqueness}, problem \eqref{BVP0} has a unique positive solution $u_p$ (which is the least-energy one). Without loss of generality, we may suppose that $u_{p_0}$ is still the unique positive solution of \eqref{BVP0}. Let  $S := \{ p \in (1, p_0] \,  \big| \,  u^{\frac{1-p}{2}} \text{ is not strictly convex in } \Om \}$ and set $p^* := \inf S$ and $v^*:= \displaystyle u_{p^*}^{\frac{1-p^*}{2}}.$ Then $1+\varepsilon \leq p^* \leq p_0$ and  $\nabla^2 v_*$ is positive semidefinite in $\Om.$ Let $x_0 \in \Om$ be such that $rank\bigl( \nabla^2 v_*(x_0) \bigr) = 1.$ Since $v^*$ satisfies \eqref{deltav_sub2}, Theorem \ref{MAIN} may be applied. Moreover, Lemma \ref{lemma_boundary_rank} gives that $\nabla^2 v^*$ is positive definite in a tubular neighborhood $\Om_\delta$ of $\partial \Om$ for some $\delta>0.$ Thus by again invoking Corollary \ref{corollary_constant_rank}, we have that $\nabla^2 v^*$ is positive definite in $\Om,$ which contradicts our hypothesis that $rank \, \nabla^2 v_* (x_0) = 1.$ Hence $S = \emptyset$ and $u^{\frac{1-p}{2}}$ is strictly convex in $\Om$ for each $1 < p \leq p_0.$ 

We now consider the case $p_0 < p \leq 3.$ We notice that as $p$ varies in the range $(1, p_0)$, the pairs $(p, u_p)$ form a continuous curve in $\R \times C^2(\bar \Om),$ which we call the least-energy branch of solutions issued from $p=1.$ As before, we assume without loss of generality that this least-energy branch extends continuously to $(p_0, u_{p_0}),$ so that $u_{p_0}$ is still the unique positive solution at $p=p_0.$ For simplicity, let us further suppose that, for some $\varepsilon > 0,$ only two continuous branches $p \mapsto (p, u_p^{(1)})$ and $p \mapsto (p, u_p^{(2)})$ are issued from $(p_0, u_0)$ and are defined for $p_0 \leq p < p_0 + \hat \varepsilon,$ with 
\begin{equation}
u_{p_0}^{(1)} = u_{p_0}^{(2)} = u_{p_0}, \qquad u_{p}^{(1)} \neq u_{p}^{(2)} \quad \text{for every } p \in (p_0, p_0+ \hat \varepsilon).
\end{equation}
Fix the first branch and set 
\begin{equation}
v_{p}^{(1)} := \left( u_p^{(1)} \right)^{\frac{1-p}{2}}, \qquad v_{p_0} :=  \left( u_{p_0}^{(1)} \right)^{\frac{1-p_0}{2}}
\end{equation}
By the continuation argument done in the case $1 < p \leq p_0$, we know that $\nabla^2 v_{p_0} > 0$ in $\Om.$ We may thus restart the continuation argument there considered for this first branch. Let $p \downarrow p_0$ along the first branch, and note that $u_p^{(1)} \to u_0$ in $C^2(\bar \Om).$  The continuation argument may now be pushed through the first branch. Indeed, let 
\begin{equation}
S:= \{ p \in (p_0, p_0+ \hat \varepsilon) \; \big| \; \left( u_{p}^{(1)}\right)^{\frac{1-p}{2}} \, \text{is not strictly convex in } \Om \},
\end{equation} 
and set $p^* := \inf S$ and $v_{p^*}^{(1)} := \left( u_{p^*}^{(1)} \right)^{\frac{1-p^*}{2}}.$  By letting $0 < \delta < \hat \varepsilon$ be sufficiently small, $\nabla^2 v_{p}^{(1)} > 0$ in $\Om$ for every $p \in (p_0, p_0 + \delta).$ Then it must be that $p_0 + \delta \leq p^* < p_0+ \varepsilon$ and $\nabla^2 v_{p^*}^{(1)}$ is positive semidefinite in $\Om.$ As before, let $x_0 \in \Om$ be such that $\operatorname{rank} \left( \nabla^2 v_{p^*}^{(1)} \right) = 1.$ By Theorem 5.1, Lemma 4.1 and Corollary 5.2, it must be that $\nabla^2 v_{p^*}^{(1)}$ is positive definite in $\Om,$ which yields a contradiction. Thus $S= \emptyset$ and $\left( u_{p}^{(1)}\right)^{\frac{1-p}{2}}$ is strictly convex in $\Om$ for each $p \in (p_0, p_0 + \varepsilon).$ This finishes the argument for the first branch. Now, returning to the bifurcation at $(p_0,u_{p_0})$, the argument repeats for the second branch. If another bifurcation happens at $(p_1, u_{p_1})$ (along either of the two branches), one then has that $\nabla^2 v_{p_1} > 0$ in $\Om,$ and the continuation argument repeats once more. Furthermore, the argument can be replicated for a possible third or more bifurcation branches. More generally, provided $1 < p \leq 3,$ one concludes that $u^{\frac{1-p}{2}}$ is strictly convex along every branch which belongs to the connected component containing the least-energy branch. This finishes the proof.

 \end{proof}

\begin{proof}[Proof of Corollary \ref{corollary_power}.]
We treat the cases $0\leq p <1$ and $1<p\leq 3$ jointly, with the understanding that, for the superlinear case $1<p\leq 3,$ the following argument is valid for any solution $u$ in the connected component of the least-energy branch of solutions. Recall that $v=u^{\frac{1-p}{2}},$ which is smooth in $\overline \Om$ since $u$ is smooth in $\overline \Om.$

We first prove that $u$ has a unique nondegenerate critical point. Clearly $u$ attains its maximum at some interior point of $\Om$ since $u\big|_{\partial \Om}=0,$ $u>0$ in $\Om$ and $\Om$ is compact. Then $u$ has at least one interior critical point. Also, the critical points of $u$ and $v:=u^{\frac{1-p}{2}}$ coincide since $\nabla v = \frac{1-p}{2} u^{\frac{-p-1}{2}} \nabla u$ and $p \in [0,1)\cup(1,3]$. Let $x_0 \neq x_1$ be two critical points of $v$. Connect them by a geodesic segment $\gamma,$ with $\gamma(0) = x_0, \gamma(1) = x_1.$ By the uniform convexity, $\gamma$ is entirely contained in $\Om.$ Define $f(t):=(v \circ \gamma)'(t) = \langle \nabla v\left( \gamma(t) \right), \gamma'(t) \rangle,$ and note $f'(0)=f'(1)=0.$ By Theorem \ref{power} \ref{sublinear}, $f$ is strictly concave on $[0,1]$ for $0 \leq p <1.$ Hence $f'$ is strictly decreasing and thus $f'(0) > f'(1),$ which is a contradiction. Similarly, by Theorem \ref{power} \ref{superlinear}, $f$ is strictly convex on $[0,1]$ for $1 < p \leq 3.$ Then $f'$ is strictly increasing and $f'(0)<  f'(1),$ again a contradiction. Therefore $v$ has a unique critical point $x_0.$  Now, because $\nabla u(x_0)=0,$ then $\nabla^2 v(x_0)=\frac{1-p}{2}u^{\frac{-p-1}{2}}(x_0) \nabla^2 u(x_0).$ By Theorem \eqref{power} \ref{sublinear} we have $\nabla^2 v(x_0)<0$ for $0\leq p <1,$ and thus $\nabla^2 u(x_0) <0.$
On the other hand,  Theorem \eqref{power} \ref{superlinear} gives that $\nabla^2 v(x_0)<0$ for $0< p \leq 3,$ and thus $\nabla^2 u(x_0) >0.$ In both cases, the unique critical point $x_0$ is nondegenerate. \\

We now prove that $u$ has strictly convex superlevel sets for $0 \leq p \leq 3.$ Let $x_0$ be the unique maximum point of $v.$ Then each $c < v(x_0)$ is a regular value of $v$ and $\partial\Sigma^c=\{x\in\Om  \, \big| \, v(x)=c\}$ is a smooth curve. Fix a $p \in \partial \Sigma^c$ and let $\nu(p) := -\frac{\nabla v(p)}{|\nabla v(p)|}$ be the outward unit normal at $p$ and $\tau(p)$ be a unit tangent vector at $p.$ The geodesic curvature at $p$ of $\partial \Sigma^c$ with respect to $\nu$ writes as $\kappa_g(p) = - \frac{\nabla^2 v (\tau, \tau) (p)}{|\nabla v(p)|}.$ 

\smallskip

• $0 \leq p \leq 1$ \\
By Theorem \ref{power} \ref{sublinear}, $v$ is strictly concave in $\Om$ and thus $\nabla^2 v (\tau, \tau) (p) < 0$ in $\Om.$ By the compactness of $\partial  \Sigma^c,$ there exists $\lambda_0 > 0$ such that $\kappa_g \geq \frac{\lambda_0}{\sup_{\Sigma^c} |\nabla v|} > 0$ on $\partial \Sigma^c.$ Thus $\partial\Sigma^c$ has strictly positive geodesic curvature with respect to $\nu$ and the superlevel set $\Sigma^c := \{x \in \Om   \setminus \{ x_0\} \; \big| \; v(x) \geq c\}$ is strictly convex. Set $a := g^{-1}(c).$ Since the map $s \mapsto s^{\frac{1-p}{1}},$ $s>0,$ is strictly increasing for $0 \leq p <1,$ then the superlevel set $\{ u(x) \geq a \}$ coincides with $\Sigma^c.$ By repeating the argument to each regular value $c$ of $v$ there follows the strict convexity of the superlevel sets of $u$. 

\smallskip

• $1 \leq p \leq 3$ \\
By Theorem \ref{power} \ref{superlinear}, $v$ is strictly convex in $\Om,$ hence $\nabla^2 v (\tau, \tau) (p) > 0$ in $\Om.$ Thus the sublevel set $\Sigma_c := \{x \in \Om  \setminus \{ x_0\} \, \big| \, v(x) \leq c\}$ is strictly convex, and the superlevel sets of $u$ coincide with the sublevel sets $\Sigma_c$ of $v$ (since now $s \mapsto s^{\frac{1-p}{2}},$ $s>0,$ is strictly decreasing for $1<p\leq 3$). Again by repeating the argument to each regular value $c$ of $v$, the strict convexity of the superlevel sets of $u$ follows.
\end{proof}

\bigskip

\section{Generalizations and Further Directions} \label{section_generalization}

In this section we address some natural directions for future research.

\subsection{Relaxation of the uniform convexity hypothesis}
The conclusions in Theorem \ref{power} and Corollary \ref{corollary_power} still hold if the \emph{uniform convexity} assumption on the smooth domain $\Om \subset \S^2$ is weakened to a \emph{convexity} one. This can be handled by an approximation argument. As explained in Section \ref{section_preliminaries}, the smooth convex domain $\Om \subset \S^2$ may be identified with a bounded, smooth convex domain $\widetilde \Om \subset \R^2.$ A standard argument allows one to construct a sequence of smooth uniformly convex domains $(\Om_t)_{t \in [0,1)} \subset \S^2$ with $\Om _t \subset \Om _{s}$ whenever $0 \leq t < s <1,$ $\bigcup_{t \in [0,1)} \Om_t=\Om$ and $\partial \Om_t \to \partial \Om$ as $t \uparrow 1.$ 

For $0 \leq p <1,$ let us recall that in Theorem \ref{theorem_uniqueness} we obtained that problem \eqref{BVP0} admits a unique positive solution under a convexity assumption on the domain. This allows one to adapt the argument of Sakaguchi (p. 410, \cite{SAKAGUCHI}), together with the deformation method of Caffarelli--Friedman \cite{CAFFARELLI-FRIEDMAN} and Korevaar--Lewis \cite{KOREVAAR-LEWIS}. Let $\varphi_t \in H_0^1(\Om_t)$ be a minimizer of 
\begin{equation} \label{variational_problem}
m_p(\Om_t):= \inf \left\{ \int_{\Om_t} |\nabla \psi|^2 dV_g \; \bigg| \;  \psi \in H_0^1(\Om_t), \, \psi \geq 0, \int_{\Om_t}  \psi^{p+1} dV_g = 1 \right\}.
\end{equation}
By a standard variational argument, one can show that $m_p(\Om_t) \to m_p(\Om_s)$ and $\varphi_t \to \varphi_s$ locally unfirmly in $\Om_s$ as $t \uparrow s,$ where $\varphi_s$ minimizes $m_p(\Om_s).$ Then setting $u_t:=m_p(\Om_t)^{\frac{1}{p-1}}\varphi_t,$ $u_s:= m_p(\Om_s)^{\frac{1}{p-1}}\varphi_s,$ gives that $u_t, u_s$ are respectively the unique positive solutions of \eqref{BVP0} in $\Om_t$ and $\Om_s,$ and $u_t \to u_s$ locally uniformly in $\Om_s$ as $t \uparrow s.$ By standard elliptic estimates, this improves to $u_t \to u_s$ in $C^\infty_{loc}(\Om_s).$ By setting $v_t:= u_t^{\frac{1-p}{2}},$ $v_s:= u_s^{\frac{1-p}{2}},$ one then obtains $v_t \to v_s$ in $C^\infty_{loc}(\Om_s).$ Let 
\begin{equation}
S:= \{ s \in [0,1]: \nabla^2 v_t \text{ is negative definite in } \Om_t \text{ for every } 0 \leq t < s \}.
\end{equation}
By Theorem \ref{power} \ref{sublinear}, $[0,1) \subset S.$ The local $C^\infty$ convergence in $\Om$ gives the negative semidefiniteness of $\nabla^2 v$ in $\Om,$ and a contradiction argument based on Theorem \ref{MAIN} then forces $1 \in S.$ Thus $v= u^{\frac{1-p}{2}}$ is strictly concave in the convex domain $\Om \subset \S^2.$

For $1<p\le3$, the approximation argument is more delicate. Since positive solutions of \eqref{BVP0} are not known to be unique in general, the argument above, which identifies the limit of the variational problem via the uniqueness of the solution, is not directly applicable. A different approach to the treatment of this case is currently being investigated by the authors.

\subsection{Generalization to $\S^n$}
A natural extension of the work here presented is the extension of Theorem \ref{power} and Corollary \ref{corollary_power} to uniformly convex domains $\Om \subset \S^n, n \geq 3.$  As in the two-dimensional case, the strategy here relies on the ingredients \ref{ing1}--\ref{ing4} described in the Introduction. 

The main difficulty appears in \ref{ing2}. For $n \geq 2,$ the identification (via an isometry) of a convex domain $\Om \subset \S^n$ with a bounded convex domain of $\widetilde \Om \subset \R^n$ endowed with a conformal metric is still possible. By the formula for the Laplace-Beltrami operator on $\S^n$ under a conformal change of metric $\tilde g = \rho^2 g,$ the equation $-\Delta_g u = u^p$ in $\Om \subset \S^n$ transforms into a more general Schrödinger equation $-\Delta_{\tilde g} u + V u= \rho^2 u^p$ in $\widetilde \Om \subset \R^n,$ where $V = \frac{n-2}{2\rho} \left( \Delta_g \rho + \frac{n-4}{2 \rho} |\nabla \rho|^2 \right).$ In the case $n=2,$ the potential $V \equiv 0$ and the equation rewrites as \eqref{eq_stereo}. Thus for $n \geq 3$ one could either adapt the uniqueness argument for $p$ sufficiently close to 1 done in Section \ref{section_uniqueness} for this more general operator (on $\R^n$), or by working directly on $\S^n.$ Moreover, note that \ref{ing1} is readily available since Proposition \ref{log_eigen} holds on $\S^n.$ As for \ref{ing3} and \ref{ing4}, they can be extended to $\S^n$ without substantial obstructions.

\bibliographystyle{abbrv}
\bibliography{GPR} 

\end{document}